\documentclass[paper]{AUXILIAIRE/classes/elsarticlePLF}

\usepackage{amsmath,amssymb,a4wide}
\usepackage{graphicx}
\graphicspath{{AUXILIAIRE/}}
\usepackage{epsf, subfigure, verbatim}
\usepackage{psfrag} 
\numberwithin{equation}{section}
\newtheorem{theorem}{Theorem}[section]
\newtheorem{lemma}[theorem]{Lemma}
\newtheorem{corollary}[theorem]{Corollary}
\newtheorem{definition}[theorem]{Definition}

\newtheorem{proposition}[theorem]{Proposition}
\newtheorem{example}[theorem]{Example}

\makeatletter
\renewcommand{\section}{\@startsection{section}{1}{\z@}%
  {14pt plus 6pt minus 3pt}%
  {9pt plus 6pt minus 3pt}%
  {\normalsize\bfseries\boldmath}}
\makeatother

\newproof{proof}{Proof}

\newcommand \bei {\begin{itemize}}
\newcommand \eei {\end{itemize}}
\newcommand \be {\begin{equation}}
\newcommand \bel {\begin{equation}\label}
\newcommand \ee {\end{equation}}

\newcommand \intconstcurv {\mathcal{I}}

\newcommand{\coloneqq}{\mathrel{:=}}

\newcommand \RR {\mathbb R}
\newcommand \eps \epsilon

\begin{document}

\title{Diffusive-dispersive traveling waves and kinetic relations. 
\\
Generic flux with multiple inflection points}

 \author{Marc Laforest\textsuperscript{a} and Philippe G. LeFloch\textsuperscript{b}}

\address[Laf]{Polytechnique Montr\'eal, d\'epartement de math\'ematiques et de g\'enie industriel, 
\\
C.P. 6079, succ. centre-ville,
Montr\'eal, Qu\'ebec, H3C 3A7, Canada. Email: marc.laforest@polymtl.ca
\vskip.2cm
}

\address[LeF]{Laboratoire Jacques-Louis Lions and Centre National de la Recherche Scientifique, 
\\
Sorbonne University, 
4 Place Jussieu, 75252 Paris, France. Email: contact@philippelefloch.org}

\begin{abstract}
We consider the broad class of scalar hyperbolic conservation laws with generic non-convex fluxes, which may thus have an arbitrary finite number of inflection points. We investigate the properties of classical and nonclassical shock waves realized as traveling-wave limits of a linear diffusive-dispersive regularization and provide a phase-plane classification beyond the case of a single inflection point. Within the speed range of the generalized Maxwell construction, our main result classifies the left outgoing profiles by their limiting equilibria. In an elementary bowtie configuration, we prove the existence of nonclassical shock profiles with arbitrarily many oscillations. For nested configurations, we order the saddle connections on each entry interval, taking account of inaccessible saddles. We show that the number of oscillations present in the non-classical shock depends on the diffusion-dispersion ratio, and tends to infinity along the elementary families of saddle connections as this ratio tends to zero. Our construction is based on a generalization of Maxwell's equal area rule, which provides us with an invariant region in the phase space for the study of diffusive-dispersive traveling waves.
\end{abstract}

\date{March 2016}

\maketitle

{\small
\setcounter{tocdepth}{1}
\tableofcontents
}


\section{Introduction}

Our purpose in this paper is the study of nonclassical entropy solutions to the following initial-value problem associated 
with a scalar conservation law in one space variable:
\be
\partial_t u + \partial_x f(u) = 0,\label{c1.1}
\ee
\be
u(x,0)  = u_0(x).
\label{c1.2}
\ee
Here, the initial data $u_0: \RR \to \RR$ is an integrable function with bounded total variation, while the flux-function $f:\RR\to \RR$ is a smooth and non-convex function which may have multiple inflection points. 

The standard approach to defining weak solutions to \eqref{c1.1}-\eqref{c1.2} is by adding a vanishing viscosity to the right-hand side of \eqref{c1.1}. However, physically realistic regularizations may involve both diffusive and dispersive terms
such as (with initial and boundary conditions that converge appropriately)
\begin{equation} 
\label{eq:lineardd} 
u^\eps_t+f(u^\eps)_x= \eps \alpha \, u^\eps_{xx}+\eps^2 u^\eps_{xxx},
\end{equation}
where $\alpha>0$ is a fixed constant and $\eps \to 0$. Traveling-wave limits yield entropy-admissible shock waves whose selection requires not only an entropy inequality (see below) but also a kinetic relation imposed on undercompressive discontinuities \cite{JMS,LeFlochbook}. LeFloch \cite{LeFlochbook} together with collaborators have established a general theory based on prescribing a kinetic function encoding
admissibility as limits of general regularizations and have shown existence of nonclassical weak solutions for \eqref{c1.1}-\eqref{c1.2}. 

The main assumption made in \cite{LeFlochbook} is that the flux-function admits a single inflection point (or more generally, as far as hyperbolic systems are concerned, Lax's genuine nonlinearity condition is assumed to be lost on each wave curve at one point at most). Our purpose in the present paper is to treat flux-functions with an arbitrary number of inflection points. We concentrate our attention on scalar conservation laws and on the construction of nonclassical traveling waves. 

The first study of undercompressive traveling waves for the cubic flux $f(u)=u^3$ with linear diffusion and dispersion was carried out by Jacobs, McKinney, and Shearer~\cite{JMS}.  Bedjaoui and LeFloch~\cite{BedjaouiLeFloch1,BedjaouiLeFloch3,BedjaouiLeFloch5} extended this analysis to general flux-functions with one inflection point and generalized the analysis to singular regularization. To these regularizations, a kinetic function characterizing the family of all traveling waves was introduced and its asymptotic properties were studied extensively.  Bedjaoui, Chalons, Coquel, and LeFloch~\cite{BedChaCoqLeF}  analyzed non-monotonic traveling waves and kinetic functions for a van der Waals system with nonlinear viscosity and capillarity. Their results motivate us in the analysis developed here. Cf.~also the reviews~\cite{LeFlochreview,ElHoeferShearer2017}. 

In this article, we analyze the traveling-wave equation for linear diffusion-dispersion regularizations. As usual, the zero diffusion separatrix curves bound absorbing regions. We prove a signed variation formula for connections between saddles, control the passage near the target saddle, and classify the parameter set into open \emph{attraction intervals} and an at most countable \emph{critical set}. The elementary bowtie construction provides sequences with prescribed numbers of turns. For nested configurations, opposite variations on a common energy section order the connections on every entry interval; descent through a finite hierarchy of bubbles classifies all outgoing profiles, while allowing inaccessible saddles.

An outline of this paper is as follows. 
In Section~\ref{sec:2}, we briefly review some elements of the theory of conservation laws and 
present key properties of the entropy dissipation. Section~\ref{sec:3} considers entropy solutions realized as limits of vanishing diffusion-dispersion
regularizations and the problem under consideration is reduced to the study of 
trajectories in a phase space. Section~\ref{sec:4} develops the generalized Maxwell construction and studies the geometry, convergence, and parameter dependence of the traveling-wave trajectories. Our main results are presented in Section~\ref{sec:general}: building on observations made in the previous sections, we prove that nonclassical traveling waves, connecting a left-hand state equilibrium to
a saddle equilibrium, arise for a countable set of critical values of the diffusion-dispersion ratio. Section~\ref{sec:6} presents some numerical results illustrating our theoretical conclusions, while in Section~\ref{sec:conclusion} we conclude with some remarks.  

%
%

\section{Generic flux and entropy dissipation}
\label{sec:2}

\subsection{Background material}

Recall that the conservation law \eqref{c1.1} may be endowed with a smooth function $(U,F): \RR \to \RR^2$ called 
an entropy-entropy flux pair and, for physical reasons, it may be required that any solution to \eqref{c1.1} satisfies the additional conservation law
\begin{equation}  \label{condition_entropie}
\partial_t U(u)  + \partial_x F(u) \leq 0.
\end{equation}
The pair $(U,F)$ satisfies $ F^{\prime}(u)=U^{\prime}(u)f^{\prime}(u)$ for $u\in \RR$ 
and with $U$ strictly convex and $C^2$.  For the linear regularization~\eqref{eq:lineardd}, we use the quadratic entropy $U(u)=u^2/2$, whose dissipative identity is compatible with this regularization.

Previous work has shown that the class of solutions should include solutions that are not continuously differentiable and, in fact, are discontinuous. For initial data $u_0\in L^{\infty}(\RR) \cap \operatorname{BV}(\RR)$,  weak solutions 
$ u\in L^{\infty}(\RR\times(0,+\infty))$ of the Cauchy problem \eqref{c1.1} must satisfy
\begin{equation}  \label{weakS}
\int\limits_{0}^{+\infty}\int\limits_{\RR}\left( u\partial_t\phi  + f(u)\partial_x \phi \right) dx dt+\int\limits_{\RR}\phi(x,0)u_0 dx=0,
       \text{ for all } \phi \in C_c^\infty(\RR\times[0,+\infty)).
\end{equation}
Weak solutions of the Cauchy problem \eqref{c1.1}-\eqref{c1.2} are generally non-unique. To single out the solution of interest, we restrict attention to weak solutions satisfying the inequality
\begin{equation}  \label{theta_ineg_entropie}
\int\int_{\RR \times \mathbb{R^+}} \left( U(u)\partial_t\phi  + F(u)\partial_x \phi \right) dx dt\geq 0,\quad \phi \in  C_c^\infty(\RR\times(0,+\infty)), \quad \phi\geq 0.
\end{equation}
In this context, the simplest initial value problem is obtained when the initial data consists of two constant states denoted by $u_\pm$, i.e. 
\begin{gather}\label{Riem_cond}
u(x,0)  = \begin{cases}
  u_- & \text{for } x < 0, \\
  u_+ & \text{for } x > 0. 
\end{cases}
\end{gather} 
The Riemann problem consists of finding a weak solution satisfying this initial condition together with the entropy inequality above. When undercompressive nonclassical shocks are allowed, a further admissibility condition, the kinetic function, must also be supplemented. We refer to \cite{LeFlochbook} for the details and, in the rest of this paper, we focus on the issue of constructing (classical and nonclassical) traveling wave solutions to \eqref{eq:lineardd} with, as mentioned earlier, the choice of entropy $U(u)=u^2/2$.

If $u$ is a discontinuity traveling at some speed $a$, written explicitly as
\begin{gather}\label{disc_with_a}
u(x,t)  = \begin{cases}
  u_- & \text{for } x < at, \\
  u_+ & \text{for } x \geq at,
\end{cases}
\end{gather} 
then if it is a weak solution to \eqref{c1.1}-\eqref{c1.2}, the speed must be
\be
     a=a(u_+,u_-) := \dfrac{f(u_+)-f(u_-)}{u_+-u_-}.
\ee
The speed given above $a=a(u_+,u_-)$ is called the Rankine-Hugoniot shock speed
and is a function of $u_+, u_-$ with the same regularity as $f$. When such a weak solution is also entropy-admissible,
in the sense of \eqref{theta_ineg_entropie}, 
it is called a shock wave. If we have a Riemann problem for \eqref{c1.1}-\eqref{c1.2}  with $u_-<u_+$ (resp. $u_+<u_-$ ) such that $f$ is convex (resp. concave) over the interval $[u_-,u_+]$ (resp. $[u_+,u_-]$), then the classical entropy solution is the continuous rarefaction wave given by
\begin{gather}
u(x,t)  = \begin{cases}
  u_- & \text{for } x/t\leq f^\prime(u_-), \\
  (f^\prime)^{-1}(\frac{x}{t})& \text{for } f^\prime(u_-)<x/t< f^\prime(u_+), \\
  u_+ & \text{for } f^\prime(u_+)\leq x/t.
\end{cases}
\end{gather}

 
\subsection{The class of flux-functions of interest}

We begin with a definition. 

\begin{definition}\label{def-generic} 
A function $f: \RR \to \RR$ is called a \textbf{generic} flux if it is $C^3$ and admits exactly $n\geq1$ non-degenerate inflection points
$I_n<I_{n-1}< \cdots <I_1$, that is, such that 
\be
       f^{\prime\prime}(I_k) =0,  \qquad f^{'''}(I_k)\neq 0 \quad \text{ for k=1,\ldots,n.} 
\ee
Moreover, we assume that the flux is not affine on any nontrivial interval and for $n$ even 
\begin{equation}  \label{1.2}
             \lim_{u \to \pm\infty} f(u) = +\infty, \qquad \lim_{u \to \pm\infty} f'(u) = \pm\infty,
\end{equation}
while for $n$ odd,  
\begin{equation} \label{1.3}
             \lim_{u \to \pm\infty} f(u) = \pm\infty, \qquad \lim_{u \to \pm\infty} f'(u) = +\infty.
\end{equation}
\end{definition}

The term generic refers here to the nondegenerate inflection points prescribed in Definition~\ref{def-generic}.   
We restrict attention to Riemann problems with 
$I_1<u_-$ and $u_+<u_-$, as was done in \cite{LeFlochbook}. Almost identical statements exists 
when $u_-<u_+$ and $u_-<I_n$, but they will be omitted since the purpose of this research is to explain 
how LeFloch's notion of kinetic functions can be extended to generic fluxes.

Some notation will be required later on and we introduce the intervals of convexity/concavity 
$\intconstcurv_i = [I_i, I_{i-1}]$ with the understanding that $I_0 = +\infty$ and $I_{n+1}= -\infty$.

\begin{example}  \label{ex:VanDerWalls}
The classical example is the flux associated to the van der Waals equation of state
which has two inflection points but does not satisfy the asymptotic properties \eqref{1.2} and \eqref{1.3}.
Furthermore, the theory was developed for the system of isentropic gas dynamics, so
the work there isn't applicable to our study of scalar conservation laws.
Despite these differences, much of the theory in this article was generalized from the $2 \times 2$ system 
studied in \cite{BedChaCoqLeF}.
\end{example}

\begin{figure}[!htb] 
\centering
\begin{minipage}{.5\textwidth}
       \centering  \includegraphics[width=\linewidth]{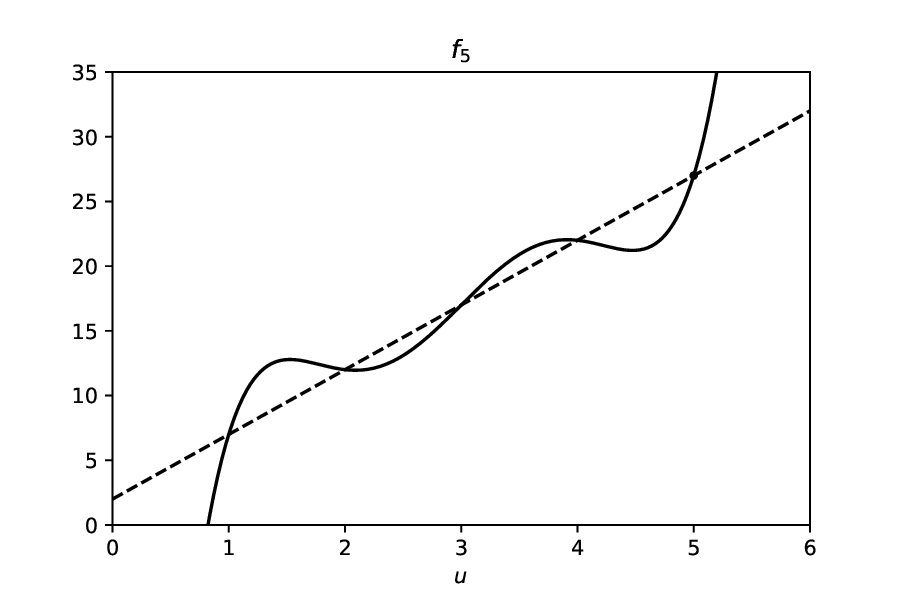}
\end{minipage}%
\begin{minipage}{.5\textwidth}
       \centering  \includegraphics[width=\linewidth]{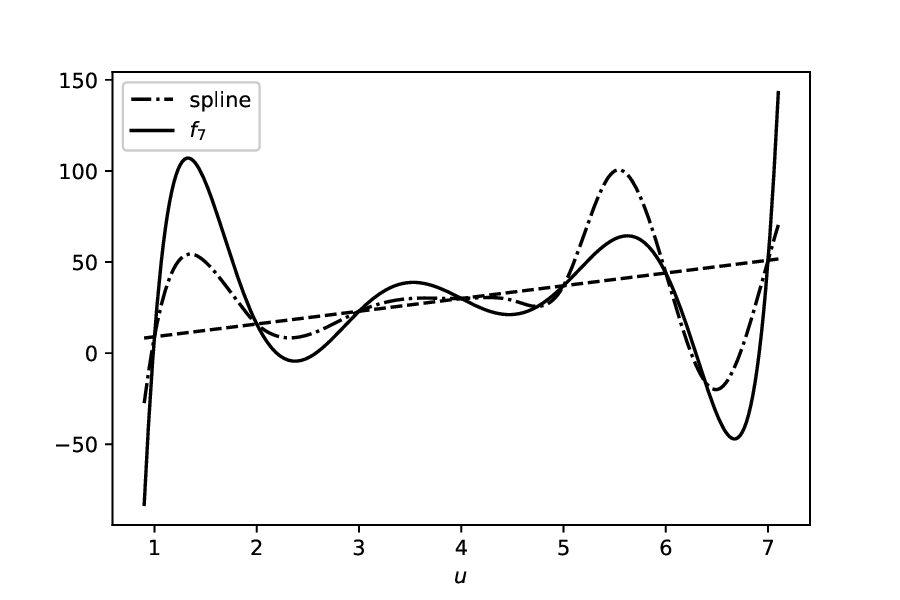}
\end{minipage} 
\caption{Graphs of different fluxes. The graph of $f_5$ is given on the left and the graphs
of $f_7$ and of the cubic spline from Example \ref{ex:fluxes2} are given on the right. 
The dotted line in each graph represents the line of slope $k$ through $(k,f_k(k))$.}
\label{fig:flux7}
\end{figure}

\begin{example}  \label{ex:fluxes}
A family of polynomial fluxes will be used to illustrate the concepts in this paper. For each 
odd $k=5,7,9, \ldots$, we define
\be
    f_k(u) := \prod_{i=1}^k (u-i) + ku + 2.
\ee
The flux $f_k$ is constructed so that the line of slope $k$ through $\big( k,f_k(k)\big)$
crosses the graph at $u=1,2,\ldots, k$, as illustrated in Figure \ref{fig:flux7}. Note
that $k$ is the degree of $f_k$, so its generic number of inflection points is $k-2$.
\end{example}

\begin{example}  \label{ex:fluxes2}
A large family of examples will also be generated as perturbations of the fluxes $f_k$ using cubic splines.
For any odd $k$, we will construct a cubic spline using the intersections of $f_k$ with the line of slope $k$
through $\big( k,f_k(k) \big)$ and choose the values of the spline at the intermediate points
$$
     1+\frac{1}{2}, 2+ \frac{1}{2}, \ldots, k - \frac{1}{2}.
$$
To complete the determination of the cubic spline, we will set the values of its derivative at $x=1$ and $x=k$ close to 
those of $f_k$. One such example will be a perturbation of $f_7$ with it's intermediate values given in Table  \ref{tab:spline1}
and it's derivatives set to be $f'(1) = 300$ and $f'(7)=200$. This spline is illustrated in Figure \ref{fig:flux7}
where it is seen that it is less oscillatory than  $f_7$ between $u=2$ and $u=5$.
\end{example}

\begin{table}[ht!]
\centering
\caption{Values determining the cubic spline of Example \ref{ex:fluxes2}.}
\label{tab:spline1}
\begin{tabular}{|c||c|c|c|c|c|c|}
\hline
$u$       & 1.5  &  2.5  &  3.5  &  4.5  &  5.5  &  6.5 \\ \hline   \hline
$f(u)$    & 50   &  10   &  30   &   29  & 100  & -20 \\ \hline
\end{tabular}
\end{table}

It is important to remark that adding an arbitrarily small bump to the flux, something like a smooth 
mollifier, can create (at least) two arbitrarily close inflection points, so that in practice 
one would expect that some inflection points are of little interest. 
In other words, when two or more inflection are present, we expect there to be scales determining the relevance 
of the inflection points. This subtle effect will appear much later in the analysis.

\subsection{Properties of the entropy dissipation}

We now discuss some properties of the entropy for generic flux. First of all, it is well-known that for weak 
solutions \eqref{weakS} formed by piecewise $C^1$ functions
with isolated discontinuities along $C^1$ curves, the weak entropy condition \eqref{theta_ineg_entropie} is satisfied 
if and only if a pointwise condition is satisfied along the curve. More specifically, if the
isolated discontinuity traveling at speed $a$ is characterized locally by
\begin{gather}\label{Aomegadisc_with_a}
u(x,t)  = \begin{cases}
  u_- & \text{for } x < x_0 +at, \\
  u_+ & \text{for } x_0+at \leq x,
\end{cases}
\end{gather}
then \eqref{theta_ineg_entropie} holds for all positive test functions if and only if
the entropy dissipation
\begin{equation}
         E(u_+,u_-)=-a(u_{+}, u_{-})\big( U(u_+)-U(u_-)\big) +\big(F(u_+)-F(u_-)\big) \label{3.2}
\end{equation}
is nonpositive at every point along the discontinuity \cite{LeFlochbook}.
The next lemma presents a form of the entropy dissipation that is
more amenable to analysis.

\begin{lemma}   \label{lem:entropyequivalent}
The entropy dissipation \eqref{3.2} is equivalent to
\begin{equation}
\begin{gathered}
E(u_+,u_-)=-\int\limits_{u_-}^{u_+}U^{\prime\prime}(v)(v-u_-)\big( a(v,u_-)-a(u_+,u_-)\big)dv.
\end{gathered}\label{entrop3.2}
\end{equation}
\end{lemma}

\begin{proof} Following the proof in  \cite[Theorem 3.1]{LeFlochbook}, we begin by rewriting the entropy dissipation 
in integral form by exploiting the relation $F' = U'f'$,
\begin{align*}
E(u_+,u_-)& =-a(u_{+}, u_{-})\big( U(u_+)-U(u_-)\big) +\big(F(u_+)-F(u_-)\big)\\
                   & = \int\limits_{u_-}^{u_+}U^\prime(v)\Big(-a(u_{+}, u_{-})+f^\prime(v)\Big)dv.
\end{align*}
Using integration by parts, we find
\begin{align*}
      E(u_+,u_-)=& \, U^\prime(v)\Big(-a(u_+,u_{-})v+f(v)\Big) \Big|_{u_-}^{u_+} 
                  -\int\limits_{u_-}^{u_+}U^{\prime\prime}(v)\Big(-a(u_{+}, u_{-})v+f(v)\Big)dv.
\end{align*}    
We now add and subtract $U^\prime(u_+)(-a(u_{+}, u_{-})u_-+f(u_-))$, simplify and obtain 
\begin{align*}
E(u_+,u_-) = & \, U^\prime(u_+)\Big(-a(u_{+}, u_{-})(u_+-u_-)+f(u_+)-f(u_-)\Big)\\
                    &\, +\Big(U^\prime(u_+)-U^\prime(u_-) \Big) \Big(-a(u_{+}, u_{-})u_-+f(u_-)\Big) \\
                    & \, -\int\limits_{u_-}^{u_+}U^{\prime\prime}(v)\Big(-a(u_{+}, u_{-})v+f(v)\Big)dv.
\end{align*}
The Rankine-Hugoniot condition implies that the first term vanishes, and so
\begin{align*}
E(u_+,u_-) = & \int\limits_{u_-}^{u_+}U^{\prime\prime}(v) \Big(-a(u_{+}, u_{-})u_-+f(u_-)\Big)dv 
                       -\int\limits_{u_-}^{u_+}U^{\prime\prime}(v)\Big(-a(u_{+}, u_{-})v+f(v)\Big)dv\\
= & -\int\limits_{u_-}^{u_+}U^{\prime\prime}(v)(v-u_-) \Big(\dfrac{-a(u_{+}, u_{-})(v-u_-)}{v-u_-}+\dfrac{f(v)-f(u_-)}{v-u_-}\Big)dv, 
\end{align*}
which is what we wanted to show.\qed 
\end{proof}

\begin{proposition}\label{thm1}
Given a generic flux function $f$ with $n$ inflection points and a state $u_- > I_1$, then the following properties hold: 
\begin{itemize}
\item[i)] $E(u_-,u_-) = 0$;
\item[ii)] $E(\cdot,u_-)$ is positive and increasing over $(u_-, +\infty)$;
\item[iii)] $E(\cdot, u_-)$ is negative and increasing for arguments sufficiently close but smaller than $u_-$;
\item[iv)] For $v\neq u_-$, the critical points of $E(\cdot,u_-)$ are exactly the points where the tangent to $f$ passes through $(u_-,f(u_-))$. Monotonicity can change only at these points.
               When $u_+ < u_-$, then
               \begin{equation}    \label{eq:signE}
                      \operatorname{sign}\Big( \partial_{u_+} E(u_+,u_-) \Big) =
                        \operatorname{sign} \Big( a(u_+, u_-) - f'(u_+) \Big). 
               \end{equation} 
\end{itemize}
\end{proposition} 
\begin{proof}
Starting from \eqref{3.2}, we compute the derivative  
\begin{align}
\partial_{u_+} E& (u_+,u_-)  \nonumber \\
 = & \, -\big(U(u_+)-U(u_-)\big)\partial_{u_+} a (u_{+}, u_{-})-a(u_{+}, u_{-})U'(u_+)+\partial_{u_+}F(u_+)  \nonumber \\
=&\big(U(u_-)-U(u_+)\big)\partial_{u_+} a (u_{+}, u_{-})-U'(u_+)\big(a(u_{+}, u_{-})-f^\prime(u_+)\big). \nonumber 
\end{align}
Since 
\begin{equation}    \label{thm1:eq1}
    \partial_{u_+}a\big(u_{+}, u_{-}\big)=\dfrac{f'(u_+)-a(u_{+}, u_{-})}{u_{+}-u_-},
\end{equation}
we conclude that
$$
      \partial_{u_+}E(u_+,u_-)=b(u_+,u_-)\partial_{u_+} a(u_{+}, u_{-}),
$$
where
$$
     b(u_+,u_-)=\big(U(u_-)-U(u_+)\big)-U'(u_+)\big(u_{-} - u_{+}\big)\,.
$$
Since $U$ is strictly convex, then $b>0$ for $ u_+\neq u_- $ and hence the monotonicity of $E$ is determined by the sign of $\partial_{u_+} a$.
The formula \eqref{thm1:eq1} shows that the critical points of $E$ are precisely the points $v$
where the slope of the tangent equals the slope of the secant joining $(v,f(v))$ and $(u_-,f(u_-))$ which is the fourth property,
and in particular the identity \eqref{eq:signE}.
Conditions \eqref{1.2} and \eqref{1.3} imply that properties {\it ii}) and {\it iii}) hold. 
\end{proof}

\begin{figure}[!htb] 
\centering
\begin{minipage}{.5\textwidth}
     \centering \includegraphics[width=1.0\textwidth]{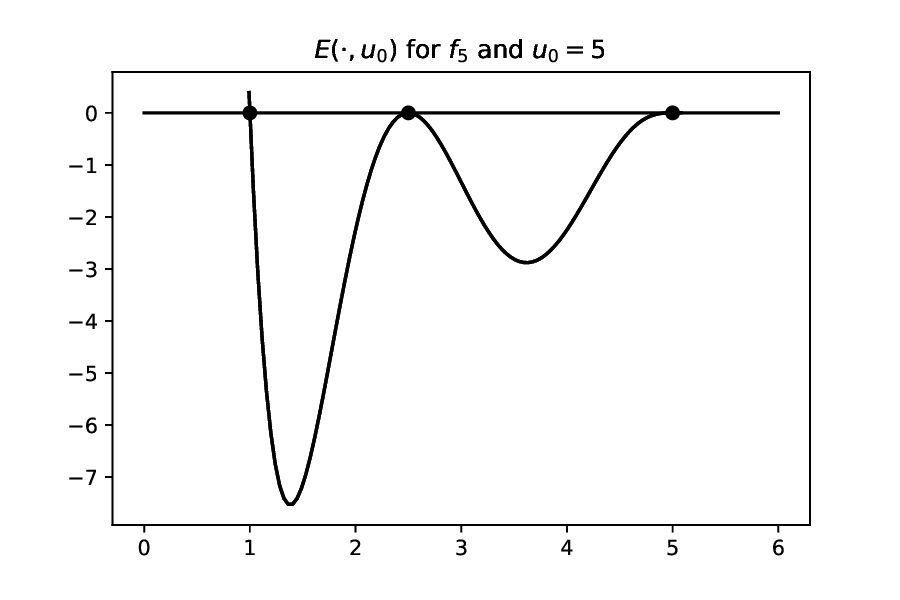} 
\end{minipage}%
\begin{minipage}{.5\textwidth}
       \centering \includegraphics[width=1.0\textwidth]{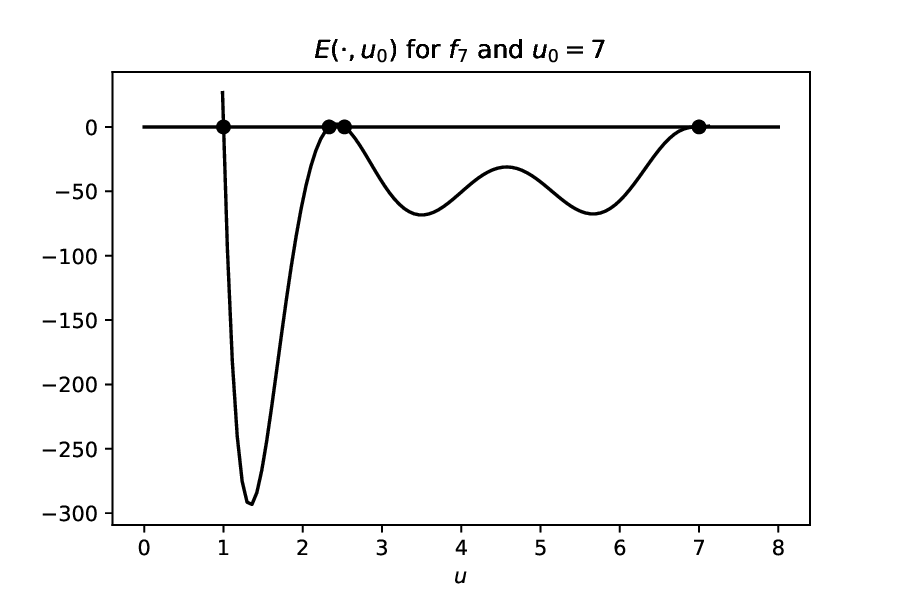}
\end{minipage} 
\caption{The entropy dissipation $E(\cdot,u_0)$ for $u_0 = 5$ for $f_5$ (on the left) and $u_0=7$ for $f_7$ (on the right) when assuming 
a quadratic entropy $U(u) = u^2/2$. For $f_5$, the entropy dissipation has a double root at $u=5/2$ while for $f_7$, 
there are roots at exactly $7/3$ and approximately $2.52695$.}
\label{fig:entropydissipationf5}
\end{figure}

\begin{example}
For the polynomial $f_5$ and the entropy $U(u) = u^2/2$, we can
compute that the polynomial entropy flux $F$ satisfying $F'(u) = U'(u) \cdot f'(u)$ is
$$
    F(u) = \frac{5}{6}u^6 - 12 u^5 + \frac{255}{4} u^4 -150 u^3 + \frac{279}{2} u^2.
$$
We can therefore analytically evaluate the entropy dissipation using the expression  \eqref{3.2}
and its graph is given in Figure \ref{fig:entropydissipationf5}. A similar computation can easily be
done for any polynomial flux combined with a polynomial entropy. 
\end{example}

\begin{example}
For the quadratic entropy $U(u) = u^2/2$, the entropy dissipation of $f_5$ has three roots $u=1, 5/2, 5$ where the second one is a 
second order root. This unexpected root is clearly visible in Figure \ref{fig:entropydissipationf5} and will be
a critical counter-example for later portions of the theory. A similar computation for the entropy dissipation of $f_7$ produces 
two additional roots, besides the obvious
roots at $u=1$ and $u=7$, as seen on the right hand side of Figure \ref{fig:entropydissipationf5}.
\end{example}

We now present some basic results concerning tangency points where the entropy dissipation
changes its monotonicity.

\begin{lemma}  \label{lem:num_tangency}
Fix $u_- > I_1$, then there exists at most one tangency point $u^{\natural}$ in each interval $\mathcal{I}_k$, that is to say
a point satisfying
\begin{equation}  \label{defn:tangency}
      f'(u^{\natural}) = \frac{f(u^{\natural}) - f(u_-)}{u^{\natural} - u_-}, 
\end{equation}
so that in total, there are at most $n$ tangency points.
\end{lemma}

\begin{proof}
We define $D(u)\coloneqq f(u)-f(u_-)-f'(u)(u-u_-)$. For $u\neq u_-$, the tangency condition is $D(u)=0$, and $D'(u)=-(u-u_-)f''(u)$. On each curvature interval to the left of $u_-$, this derivative has a fixed sign and cannot vanish identically on any subinterval, so $D$ is strictly monotone and has at most one zero. On the rightmost curvature interval, strict convexity gives $D(u)<0$ for every $u\neq u_-$. There are $n$ remaining curvature intervals, proving the bound.
\end{proof}

\begin{definition}
Assume that there exists $J \geq 1$ tangency points, denoted by 
$$
      u^{\natural}_J < u^{\natural}_{J-1} < \cdots < u^{\natural}_1,
$$ 
that is,  points at which \eqref{defn:tangency} is satisfied.
Introduce the intervals $\mathcal{I}^{\natural}_i = [u^{\natural}_i, u^{\natural}_{i-1}]$,
with the assumption that $u^{\natural}_0 = u_-$.
\end{definition}

Later, we will show (in Lemma \ref{lem:exist_tangency}) that there exists at least one tangency point, i.e. $J \geq 1$. 
At this stage of the analysis, Lemma \ref{lem:num_tangency} only provides an upper bound on the number
of tangency points.

\begin{corollary}
In each interval $\mathcal{I}^{\natural}_i$, the entropy dissipation function $E(\cdot, u_-)$ admits at most one root.
\end{corollary}

\begin{proof} By Proposition \ref{thm1}, the monotonicity of $E$ only 
changes at the endpoints of the interval $\mathcal{I}^{\natural}_i$, hence it can have at most one root
in that interval.
\end{proof}

\section{Diffusive-dispersive traveling waves}
\label{sec:3}

\subsection{The traveling wave problem}

We now derive the traveling wave equation and present conditions under which such solutions exist. 
We thus consider the conservation law with linear diffusion and dispersion \eqref{eq:lineardd}, or say  in the form 
\begin{equation}\label{defusidispersieq}
\partial_t u^{\eps}  +\partial_x f(u^{\eps}) =\eps b u^{\eps}_{xx} +\eps^2 c u^{\eps}_{xxx},
\end{equation}
in which $f:\RR\to \RR$ is a generic flux with $n$ inflection points, and the ratio 
$
    \alpha :=\dfrac{b}{\sqrt{c}}
$
is held fixed for some constants $b,c > 0$. Solutions to the regularized conservation law \eqref{defusidispersieq}
are sought here in the class of traveling profiles depending on $(x-at)/\eps$, with limits as $\eps\to0$ of the form~\eqref{Aomegadisc_with_a}.

\begin{definition} A function $u^\eps(x, t) = w(y)$ with
$y:=\alpha \dfrac{(x-at)}{\eps b}$ is called a traveling wave of the equation \eqref{defusidispersieq} 
(traveling at the speed $a$ and connecting $u_-$ to $u_+$) 
 if it is a smooth solution to \eqref{defusidispersieq} satisfying 
$$
    w(-\infty)=u_-,\qquad w(+\infty)=u_+,
$$
$$   
      \lim\limits_{|y|\to +\infty}w^\prime(y)=\lim\limits_{|y|\to +\infty}w^{\prime\prime}(y)=0.
$$
\end{definition}

When traveling waves of \eqref{defusidispersieq} exist, one is interested in their limit
when $\eps \to 0$. This is a kind of admissibility criteria for
shock waves of the conservation law \eqref{c1.1}. More general regularizations,
in particular nonlinear ones, have been studied in Bedjaoui and LeFloch \cite{BedjaouiLeFloch5,BedChaCoqLeF}
for fluxes with one inflection point and for van der Waals systems. Here we restrict the analysis to linear regularizations.

It is known that if $w$ is a traveling wave solution, then the pointwise limit away from $x=at$
\begin{gather}\label{condtravwave}
u(x,t):=\lim\limits_{\eps\to 0}w\left(\alpha \dfrac{x-at}{\eps b}\right)= 
\begin{cases}
  u_- & \text{for } x < at, \\
  u_+ & \text{for } x > at,
\end{cases}
\end{gather} 
is a weak solution to \eqref{c1.1} satisfying the entropy inequality \eqref{condition_entropie} when
the entropy is taken to satisfy
$
      U''(u) = c,
$
where the proof can be found in \cite{LeFlochbook}. In the remainder, we will therefore assume that 
the entropy is chosen in this form, which reduces to
\begin{equation}    \label{entropy:explicit}
       U(u) = c \frac{u^2}{2},
\end{equation}
with the entropy flux taken to satisfy $F' = U' \cdot f'$. Note that
the results in the previous section on entropy dissipation hold for an arbitrary entropy $U$.

Abusing our notation for $u$, we search for traveling wave solutions $u=u(y)$ of \eqref{defusidispersieq} depending on the rescaled 
spatial variable $y:=\alpha\dfrac{(x-at)}{\eps b} = \dfrac{(x-at)}{\eps \sqrt{c}}$.
Substituting $u(y)$ into \eqref{defusidispersieq}, we obtain
\begin{align*}
          \dfrac{1}{\eps \sqrt{c}}\big( -a u_y+f_y(u)\big)&=\dfrac{1}{\eps \sqrt{c}} 
                \bigg( b\eps \dfrac{1}{\eps \sqrt{c}} u_{yy}+c \eps^2 \dfrac{1}{c \eps^2}            u_{yyy}\bigg),\\
      -a u_y+ f_y(u)&  =\alpha u_{yy}+u_{yyy}.
\end{align*}
We find that a traveling wave of \eqref{defusidispersieq} connecting $u_-$ to some state $u_+$, is the solution  $ u(y)$ of the ordinary differential equation 
\begin{equation}\label{tointegrate}
        -a u_y+ f_y(u)=\alpha u_{yy}+ u_{yyy},
\end{equation}
with the boundary conditions
\begin{equation}\label{Aboudcond}
        \lim\limits_{y\to \pm\infty}u(y)=u_\pm,\quad \lim\limits_{y\to \pm\infty}u_y(y)=\lim\limits_{y\to \pm\infty}u_{yy}(y)=0.
\end{equation}
Integrating \eqref{tointegrate} from $L$ to $y$, we deduce
\begin{align*}
         \int\limits_{L}^{y}(-a u_y+ f_y(u))dy & = \int\limits_{L}^{y}(\alpha u_{yy}+ u_{yyy})dy,\\
        -a\big(u(y)-u(L)\big)+f\big(u(y)\big)-f\big(u(L)\big)&=\alpha \big(u_y(y)-u_y(L)\big)+\big(u_{yy}(y)-u_{yy}(L)\big),
\end{align*}
In view of \eqref{Aboudcond}, taking the limit as $L\to -\infty$, we obtain 
\begin{equation}\label{integtravwave}
          -a\big(u(y)-u_-\big)+f\big(u(y)\big)-f\big(u_-\big)=\alpha u_y(y)+ u_{yy}(y).
\end{equation}
Letting $y \to +\infty$, we see that the speed $a$ must satisfy 
\begin{equation}
        -a(u_+-u_-)+f(u_+)-f(u_-) = 0,
\end{equation}
which is the Rankine-Hugoniot condition. For the remainder of the analysis, the speed $a$ will therefore
be determined by the choice of $u_-$ and $u_+$.

Setting now $v= u_y$. We reformulate \eqref{integtravwave} as a differential system of two equations, 
\begin{equation}\label{defsystemdiff2}
    \dfrac{d}{dy}
     \left(\begin{matrix}
     u\\v
    \end{matrix}\right)
   =
     \left( \begin{matrix}
      v\\ -\alpha v+\big(  f(u)-au \big) -\big( f(u_-)-au_-\big)
     \end{matrix}\right).
\end{equation}
We define the function $T(u,v)$ as 
\begin{equation}\label{defKuv}
      T(u,v)=
       \left(\begin{matrix}
      v\\-\alpha v +\big( g(u,a) - g(u_-,a)\big)
     \end{matrix}\right), 
     \quad g(u,a)=f(u)-au,
\end{equation}
while the boundary conditions take the form
\begin{equation}\label{defcond}
       \lim\limits_{y\to \pm\infty}\left(\begin{matrix}
      u(y)\\v(y)
       \end{matrix}
      \right)=\left(\begin{matrix}
      u_\pm\\0
        \end{matrix}\right),
\end{equation}
with 
$
       \lim\limits_{y\to \pm\infty}v_y(y)=0.
$
It is clear that solutions to system \eqref{defsystemdiff2} are at least $C^3$, and in fact the
first component will be at least $C^4$, by the fundamental theorem for ordinary differential equations.

\subsection{Traveling waves associated with a given speed}
We want to study the asymptotic behavior of trajectories of \eqref{defsystemdiff2}. For this purpose, we consider 
the stability of the equilibria of  \eqref{defsystemdiff2}. By definition, a point $(u,v)$ is called 
an \textit{equilibrium point} of \eqref{defsystemdiff2} if  $T(u,v)=(0,0)^T$. In this
work, the equilibrium points of \eqref{defsystemdiff2} will require $v=0$ so that the term equilibrium
will be used interchangeably with the $u$-coordinate of the point.

\begin{lemma}  \label{lem:equilibrium}
Consider a generic flux function $f$ with $n$ inflection points and a state $u_-$ fixed, greater than 
the largest inflection point $I_1$. The function T in \eqref{defKuv} vanishes at the equilibrium points $(u,v)\in \RR^2$ satisfying
\begin{equation}\label{equilibria}
                          g(u, a) = g(u_-, a)\text{ and } v= 0.
\end{equation}
The equilibrium $(u_-,0)$ is always present. Every other equilibrium satisfies the Rankine--Hugoniot condition with speed $a$, 
and there are at most $n+2$ equilibria satisfying~\eqref{equilibria}.
\end{lemma}

\begin{proof}
The equilibria occur at values of $u$ and $v$ for which 
\begin{equation}\label{defsystemdiff}
       T(u,v)=
       \left(\begin{matrix}
       0\\0
      \end{matrix}\right).
\end{equation}
This implies that $v=0$ and  $g(u, a) = g(u_-, a)$ which is equivalent to $$f(u)-au = f(u_-)-au_-, $$ hence,  for $u\neq u_-$,
\begin{equation}   \label{eq:simple_equilibria}
 a = \dfrac{f(u)- f(u_-)}{u-u_-}.
\end{equation}
The function $f'-a$ is strictly monotone on each of the $n+1$ curvature intervals, so it has at most $n+1$ zeros. Rolle's theorem then bounds the number of distinct zeros of $f(u)-au-f(u_-)+au_-$ by $n+2$.
\end{proof}

\begin{example}
The fluxes $f_k$ of Example \ref{ex:fluxes} have particularly simple equilibrium points when the slope is $k$.
First of all, we note that the condition~\eqref{equilibria} is equivalent to~\eqref{eq:simple_equilibria}  for $u\neq u_0$, while $u_0$ itself is always 
an equilibrium. For $u_0=k$ and the slope $a=k$,
the equilibrium at $u$ satisfy \eqref{eq:simple_equilibria}
$$
     k (u-k) = \Big( \prod_{i=1}^k (u-i) + ku+2 \Big) - \Big( \prod_{i=1}^k(k-i) + k \cdot k +2  \Big),
$$
which can be simplified to
$
    0 = \prod_{i=1}^k (u-i).
$
This shows that the equilibria are precisely $u=1,2,\ldots,k$. Points are usually counted from right to left and
because we are focusing on travelling waves with a large left-hand state $u_0$, we will usually write 
$$
   u_0 =k, u_1 = k-1, \ldots, u_{k-1} = 1.
$$
\end{example}

To exhibit nonclassical trajectories, consider the geometric situation where there are exactly $K+1$  equilibria $(u_i,0)$, 
$i = 0, 1, 2,\ldots, K$  satisfying \eqref{equilibria}. We number these points from right to left 
\begin{equation}     \label{defn:equilibriapoints}
       u_{K}< \cdots < u_1<u_0:=u_-.
\end{equation}
If $(u_i, 0)$ is an equilibrium point, the Jacobian matrix $DT(u_i)$ is given by
\begin{equation}
   DT(u_i)=\left(\begin{matrix}
   0\quad \quad \quad \quad \quad \quad 1\\
    (f^\prime(u_i)-a)\quad -\alpha
    \end{matrix}\right).
\end{equation}
The characteristic equation of $DT(u_i)$ 
\begin{equation}
      \operatorname{det}\big( DT(u_i)-\mu I \big)=\mu^2+\alpha \mu - (f^\prime(u_i)-a)=0
\end{equation} 
which admits two roots. They are found to be
\begin{align*}
\underline{\mu}_i=  \, -\dfrac{\alpha}{2} - \sqrt{\left(\dfrac{\alpha}{2}\right)^2+ (f^\prime(u_i)-a)}, 
\qquad\quad 
\overline{\mu}_i= \, -\dfrac{\alpha}{2}+ \sqrt{\left(\dfrac{\alpha}{2}\right)^2+ (f^\prime(u_i)-a)}.
\end{align*}
The two corresponding eigenvectors  are
\begin{equation}  \label{eigenvectors}
     \underline{r}_i = 
            \left(\begin{matrix}
       1 \\ \underline{\mu}_i
      \end{matrix}\right), 
 \quad \qquad  
      \overline{r}_i =              \left(\begin{matrix}
       1 \\ \overline{\mu}_i
      \end{matrix}\right).
\end{equation}

\begin{lemma}\label{valpropres}
For $i=0,\ldots,K$, with the state $u_i$ held fixed and $(\alpha/2)^2+f'(u_i)-a>0$, we have
$$
\dfrac{\partial \underline{\mu}_i }{\partial a}(u_i;a,\alpha)>0,\quad
\dfrac{\partial \underline{\mu}_i }{\partial \alpha}(u_i;a,\alpha)<0, 
\quad \dfrac{\partial \overline{\mu}_i }{\partial a}(u_i;a,\alpha)<0.
$$
Moreover, under the assumption $f'(u_i)-a>0$, we have
$$
\dfrac{\partial \overline{\mu}_i }{\partial \alpha}(u_i;a,\alpha)<0.
$$
\end{lemma}

\begin{proof} This requires only a simple verification based on the explicit formulas for the roots.  \end{proof}

\begin{lemma}
Between each consecutive pair of equilibrium points $u_i,u_{i+1}$ with $i=1,\ldots,K-1$, there exists
at least one tangency point $u^{\natural}$, that is a point satisfying
$$
             f'(u^{\natural}) = \frac{f(u^{\natural}) - f(u_-)}{u^{\natural} - u_-}.    
$$
In particular, the number of equilibrium points is bounded, i.e. $K \leq n+1$.
\end{lemma}

\begin{proof}
For $i\geq1$, the function $u\mapsto a(u,u_0)$ is smooth on $[u_{i+1},u_i]$ and has the same value $a$ at both endpoints. Rolle's theorem gives an interior point where its derivative vanishes, which is exactly the stated tangency condition. Thus there are at least $K-1$ tangencies. Lemma~\ref{lem:num_tangency} gives $K-1\leq n$, proving $K\leq n+1$.
\end{proof}

\begin{definition}\label{def21} Consider a differential system of the form
$$
        \frac{d}{dy} \left( \begin{array}{c} u \\ v \end{array} \right) = S(u,v) 
$$
where $S:\mathbb{R}^2 \rightarrow \mathbb{R}^2$ is a smooth mapping. Let $(u^*, v^*)\in \RR^2$ be an equilibrium
point, that is, a root of $S(u^*, v^*) = 0$. Denote by $ \underline{\mu} =\underline{ \mu}(u^*, v^*)$ and $\overline{\mu} = \overline{\mu}(u^*, v^*)$ the
two (real or complex) eigenvalues of the Jacobian matrix of $S$ at $(u^*, v^*)$. Then, the
equilibrium $(u^*, v^*)$ is called
\begin{itemize}
\item a {\bf stable point} if $Re(\underline{\mu})$ and $Re(\overline{\mu})$ are both negative,
\item a {\bf saddle point} if  $Re(\underline{\mu})$ and $Re(\overline{\mu})$ have opposite sign,
\item or an {\bf unstable point} if  $Re(\underline{\mu})$ and $Re(\overline{\mu})$ are both positive.
\end{itemize}
Moreover, a stable or unstable point is called a {\bf node} if the eigenvalues are real and
a {\bf spiral} if they are complex conjugate.
\end{definition}

For the fixed state $u_0=u_-$ and speed $a$, we introduce the potential
\be
\label{defn:G}
G(u;u_0,a)\coloneqq\int_{u_0}^{u}\big[g(z,a)-g(u_0,a)\big]\,dz.
\ee
Here, $g(u,a)=f(u)-au$ is defined in the equation~\eqref{defKuv}. When $u_0$ and $a$ are fixed, we write $G(u)=G(u;u_0,a)$ and define the energy
\be
\label{eq:lyapunov-correction}
H(u,v)\coloneqq\frac{v^2}{2}-G(u),\qquad \frac{dH}{dy}=-\alpha v^2.
\ee
The second identity holds along solutions of the system~\eqref{defsystemdiff2}, since $u_y=v$ and $v_y=-\alpha v+G'(u)$. Thus $H$ is conserved for $\alpha=0$ and is a Lyapunov function for $\alpha>0$.

Since we consider the asymptotic behaviour  $u\to u_+$ as $y\to+\infty$ and $u\to u_-$ as $y\to-\infty$, according to the above definition, one can easily verify the following result.

\begin{lemma}    \label{lem:type-equilibrium}
For each fixed value of $u_-$, $a$, each equilibrium point $(u_i, 0)$, and $\alpha>0$, one then have the following two observations
\begin{itemize}
\item if $f^\prime(u_i)-a < 0$, then $(u_i, 0)$ is a stable point,
\item if $f^\prime(u_i)-a > 0$, then  $(u_i, 0)$ is a saddle point.
\end{itemize}
Furthermore, in the cases that $f^\prime(u_i)-a< 0$ we have the additional result:
When $f^\prime(u_i)-a \geq -(\alpha/2)^2$ the equilibrium is a node, and is
a spiral otherwise.
When $\alpha=0$ and $f'(u_i)<a$, the equilibrium is a center, as follows from the conservation of $H$ in the equation~\eqref{eq:lyapunov-correction} and the inequality $G''(u_i)=f'(u_i)-a<0$: the nearby regular energy levels surrounding $(u_i,0)$ are periodic trajectories.
\end{lemma}


\begin{lemma}   \label{remark1}
Consider a generic flux function $f$ with $n$ inflection points and the related system 
of differential equations \eqref{defsystemdiff} with $K+1$ equilibria $(u_i,0)$, $i=0,1,\ldots,K$, satisfying \eqref{equilibria}.
Assume $a<f'(u_0)$ and $f'(u_i)\neq a$ at every equilibrium, and let $\alpha>0$.
\begin{itemize}
\item[i)] The first  equilibria $u_0$ is a saddle point.
\item[ii)] The $K+1$  equilibria alternate between a saddle point and a stable point (either node or spiral).
\end{itemize}
\end{lemma} 

\begin{proof}
Set $q(u)=f(u)-f(u_0)-a(u-u_0)$. Its zeros are simple, and $q'(u_0)=f'(u_0)-a>0$. Between two consecutive zeros $q$ has a constant sign, so the signs of its derivatives at successive zeros alternate. Lemma~\ref{lem:type-equilibrium} gives the stated saddle--stable alternation.
\end{proof}

For the indexed separatrices below, assume $a<f'(u_0)$ and $f'(u_i)\neq a$ at every equilibrium. Additional speed restrictions will be specified after the Maxwell construction.

We are looking for trajectories connecting $(u_0, 0)$ to one of the other $K$ equilibrium points. We therefore
introduce some notation for the trajectories outgoing or incoming from these points. At each point $(u_i,0)$,
we denote the four quadrants centered at that point by
\begin{alignat*}{2}
Q_i^1 &= \bigl\{(u,v): u \geq u_i,\ v \geq 0\bigr\},
\qquad & Q_i^2 &= \bigl\{(u,v): u \leq u_i,\ v \geq 0\bigr\},\\
Q_i^3 &= \bigl\{(u,v): u \leq u_i,\ v \leq 0\bigr\},
\qquad & Q_i^4 &= \bigl\{(u,v): u \geq u_i,\ v \leq 0\bigr\}.
\end{alignat*}
For each fixed $u_0$ and wave speed $a$,  and each saddle equilibrium $u_i$, we define $C^j_{i,\alpha}$ to be the trajectory (not the solution)
leaving or entering (asymptotically) the quadrant $Q^j_i$ for a fixed diffusion-dispersion ratio $\alpha$.
 
\begin{figure}[!htb]
  \psfrag{B}[][][1]{$Q^1_i$}
  \psfrag{A}[][][1]{$Q^2_i$}
  \psfrag{C}[][][1]{$Q^3_i$}
  \psfrag{D}[][][1]{$Q^4_i$}
  \psfrag{u}[][][1]{$u$}
  \psfrag{v}[][][1]{$v$}
  \psfrag{F}[][][1]{$(u_k,0)$}
\centering \includegraphics[width=0.4\textwidth]{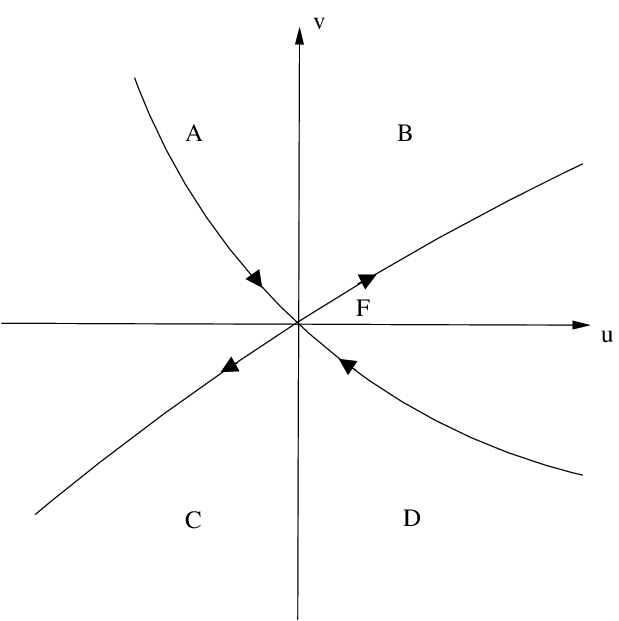}
\caption{\label{figsaddle} The four neighboring quadrants of a saddle equilibrium $(u_k,0)$.}
\end{figure}

\begin{lemma} Fix $u_0$, the speed $a$ and the diffusion-dispersion ratio $\alpha$.
Consider an equilibrium point $(u_i,0)$ of system \eqref{defsystemdiff2} and the four neighboring quadrants, then
if $(u_i,0)$ is a saddle equilibrium, there are two trajectories $C^2_{i,\alpha}$ and $C^4_{i,\alpha}$ approaching $(u_i,0)$ as  
$y\to +\infty$.  There are also two trajectories $C^1_{i,\alpha}$, $C^3_{i,\alpha}$ leaving  $(u_i,0)$ for $y \to -\infty$. 
\end{lemma}
%

\begin{proof}  The result follows immediately from the calculation of
the eigenvectors \eqref{eigenvectors} and the fact that the eigenvalues satisfy $\underline{\mu}_i < 0 < \overline{\mu}_i$ at all 
equilibrium points.
\end{proof}

\section{Waves as a function of the diffusion-dispersion ratio}
\label{sec:4}

\subsection{A generalization of Maxwell's equal area rule}

A description of the travelling waves as a function of the diffusion-dispersion ratio 
will require many intermediate technical results. First, we will introduce a generalization
of Maxwell's equal area rule that allows us to constrain the solutions to a bounded
region of phase space. Afterwards, we analyze the solutions corresponding
to zero diffusion and, with the help of a comparison result between waves with 
different diffusion-dispersion ratios, identify these as absorbing regions for 
the underlying  system of differential equations. Finally, we will conclude
with a continuity result for waves at critical diffusion-dispersion ratios, that is ratios 
for which the solutions converge asymptotically to a saddle equilibrium point. 

Our objective is to study traveling wave solutions as a function of the diffusion-dispersion ratio $\alpha$
but the trajectory of these solutions in phase-space will be delimited by the regions of the solutions
corresponding to purely dispersive waves. More specifically, as in  \cite{BedChaCoqLeF}, a generalization of 
Maxwell's equal area rule  plays a fundamental role in characterizing such traveling waves. We will therefore 
require the following fundamental assumption on the flux and the entropy to ensure 
the existence of trajectories connecting saddle points to saddle points. 

\textbf{Maxwell's equal area rule}. {\it 
We assume that the flux $f$ and the entropy $E$ are such that 
for each fixed left-hand state $u_0$ inside an open interval $\mathcal{J}_0$, there exists a largest root $\varphi^{\flat}_0(u_0) < u_0$
of the entropy dissipation 
\begin{equation}    \label{hypothesis_E}
       E \big( \varphi^{\flat}_0(u_0), u_0 \big) = 0  \qquad \text{and} \qquad \partial_{u_+} E(\varphi^{\flat}_0(u_0), u_0 \big) \neq 0.
\end{equation}} 

This is an additional hypothesis: the asymptotic conditions in Definition~\ref{def-generic} alone do not ensure that the largest nontrivial zero of $E$ is simple. The derivative condition in~\eqref{hypothesis_E} permits continuation of that zero by the implicit function theorem. We note that the entropy dissipation for $f_5$ does not satisfy this condition, as seen in
Figure \ref{fig:entropydissipationf5}. Concerning the article by Bedjaoui, Chalons, 
Coquel and LeFloch \cite{BedChaCoqLeF} which guided much of this work, 
it is more accurate to state that they did not assume Maxwell's equal area rule
but the less general condition that the function $G$ possessed a unique largest root for 
a fixed $u, u_0$ and $a$. The formalism presented here is more consistent with the original approach
of LeFloch \cite{LeFlochbook}, where the zero-diffusion kinetic function $\varphi^{\flat}_0$ plays the fundamental
role of restricting the values of the admissible left-hand states.

\begin{figure}[!htb]
  \psfrag{E}{$E(  \, \cdot, u_0 )$}
  \psfrag{u}{$u$}
  \psfrag{u0}{$u_0$}
  \psfrag{phib0}{$\varphi^{\flat}_0(u_0)$}
  \psfrag{usharp1}{$u^{\natural}_1$}
  \psfrag{usharp2}{$u^{\natural}_2$}
   \psfrag{usharp3}{$u^{\natural}_3$}
\centering \includegraphics[width=0.6\textwidth]{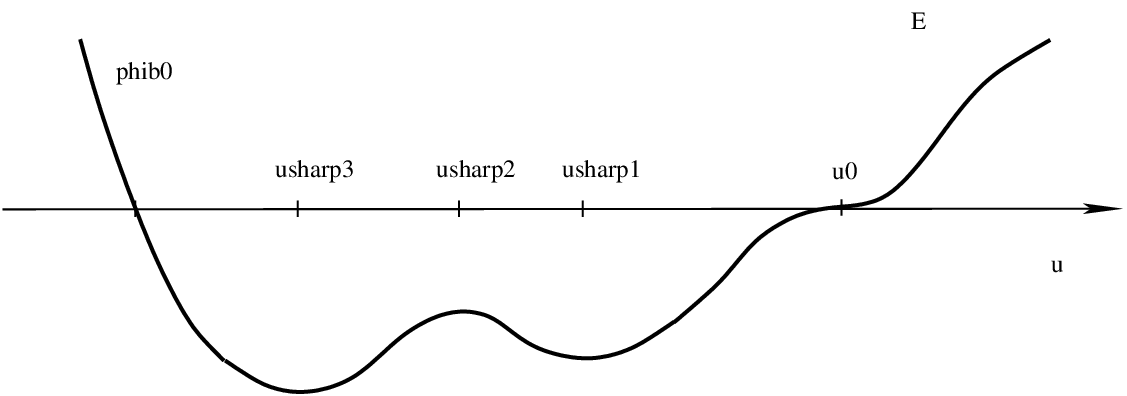}
\caption{Graph of the entropy $E(u,u_0)$ as a function of $u$ for a fixed $u_0$. The genericity condition
$\partial_{u_+} E( \varphi^{\flat}_0(u_0), u_0 \big) \neq 0$ is satisfied.}
\label{fig:entropy}
\end{figure}

\begin{lemma}
The function $\varphi^{\flat}_0$ is $C^2$ over $\mathcal{J}_0$.
\end{lemma}

\begin{proof}
The function $\varphi^{\flat}_0$ is the root of $E$, where we use its description in \eqref{entrop3.2}
as an integral depending on $f$ and $U''$ and whose
limits of integration depend on $u$ and $u_0$. The root is therefore $C^2$ since the entropy dissipation is $C^2$ 
and the derivative condition means that the intersection is transversal with respect to the upper limit of integration
and the function $a = a(u,u_0)$ in the  integrand.  
\end{proof}

To be useful, the above condition on the entropy must be related to the system \eqref{defsystemdiff2}. 
We recall the potential $G$ defined in the equation~\eqref{defn:G}, with $u_0=u_-$.
With this function, the system \eqref{defsystemdiff2} can be rewritten in the form
\be 
\aligned
u_y&=v,\\
v_y&=-\alpha v+\partial_u G(u;u_0,a), \label{Agsystemb}
\endaligned
\ee
with the boundary conditions \eqref{defcond}.
Critically, Lemma \ref{lem:equilibrium} shows that the derivative $\partial_uG(u; u_0, a)$ vanishes exactly at the equilibria $u_k,$ $k=0,1,\ldots,K,$
satisfying \eqref{equilibria}. Furthermore, we will show that for the simple entropy \eqref{entropy:explicit} the functions $G$ and $E$ are related by
\begin{equation}  \label{relation_E_G}
           c \, G \big( u;u_0,a(u,u_0) \big) = -E(u,u_0),
\end{equation}
where $c$ is the constant in the definition of the entropy. This requires a careful calculation
which begins with the formula for $E$ given in Lemma \ref{lem:entropyequivalent}.
\begin{align*}
             E(u,u_0) = & \, -c \int_{u_0}^u (v-u_0) \big( a(v,u_0) - a(u,u_0) \big) \, dv \\
             = & \, -c \int_{u_0}^u \big( f(v) - f(u_0) -a(u,u_0) (v-u_0) \big) \, dv \\
             = & \, -c \int_{u_0}^u \Big[ \big( f(v) - a(u,u_0) v\big) - \big(  f(u_0) - a(u,u_0) u_0 \big)  \Big] \, dv \\
             = & \, - c \int_{u_0}^u \Big( g\big( v,a(u,u_0) \big) - g\big( u_0,a(u,u_0) \big) \Big) \, dv  = - c \, G\big( u;u_0,a(u,u_0) \big) 
\end{align*}

%
%
\begin{lemma}\label{propG}
For $u_0 >  I_1$ and an entropy \eqref{entropy:explicit}, 
the function $G (u;u_0,a)$ has the following properties.
\begin{itemize}
\item[i)] $G(u_0;u_0;a) = 0$
\item[ii)] For all $u \in ] \varphi^{\flat}_0(u_0), u_0 [$, $G \big( u;u_0,a(u,u_0) \big) > 0 $.
\item[iii)]  The derivative with respect to $a$ is
\begin{equation}     \label{eq:derivativedaG}
       \partial_a G(u;u_0,a) = - \frac{(u-u_0)^2}{2}.
\end{equation}
\item[iv)] The derivative with respect to $u$ is
\begin{equation}     \label{eq:derivativeduG}
       \partial_u G(u;u_0,a) =  (u-u_0) \big[ a(u,u_0) - a \big]\,,
\end{equation}
from which we deduce that $G$ changes monotonicity with respect 
to $u$ only at the equilibria $u=u_k$ for some $k=0,1,\ldots,K$. To be specific, if $u_k$
is a saddle equilibrium, i.e. if $k$ is even, then $\partial_u G$ is negative over
$]u_{k+1},u_k [$ and positive over $]u_k,u_{k-1}[$.
\end{itemize}
\end{lemma}

\begin{proof}
Both the first and second properties follow from the identity \eqref{relation_E_G}, specifically because 
$\varphi^\flat_0(u_0)$ is defined to be the first root less than $u_0$. To verify the third property, first notice that
$\partial_a g(u,a) = -u$ and so
\begin{align*}
      \partial_a G(u;u_0,a) = & \, \int_{u_0}^u \big(-z -(-u_0) \big) \, dz = - \frac{(u - u_0)^2}{2}.
\end{align*}
For the fourth property, we compute the derivative
\begin{align*}
         \partial_u G(u;u_0,a) = & \, g(u,a) - g(u_0,a) \\
          = & \, \frac{f(u)-f(u_0)}{u-u_0} \cdot (u-u_0) - a( u - u_0) 
          =  \, (u-u_0) \cdot \big( a(u,u_0) - a \big).
\end{align*}
The changes in monotonicity occur at the equilibria because of Lemma \ref{lem:equilibrium}.
\end{proof}

\begin{figure}[!htb] 
\centering
\begin{minipage}{.5\textwidth}
     \centering \includegraphics[width=1.0\textwidth]{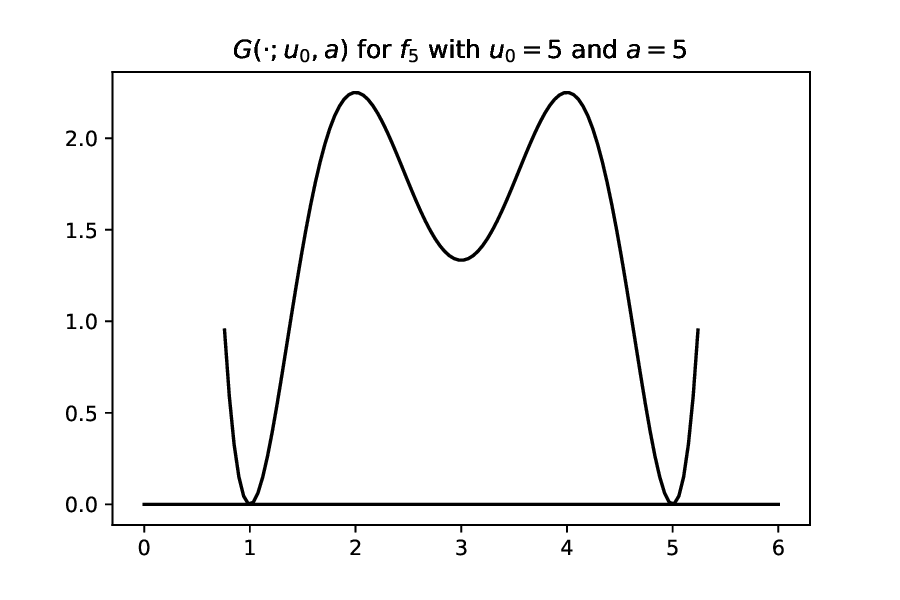} 
\end{minipage}%
\begin{minipage}{.5\textwidth}
       \centering \includegraphics[width=1.0\textwidth]{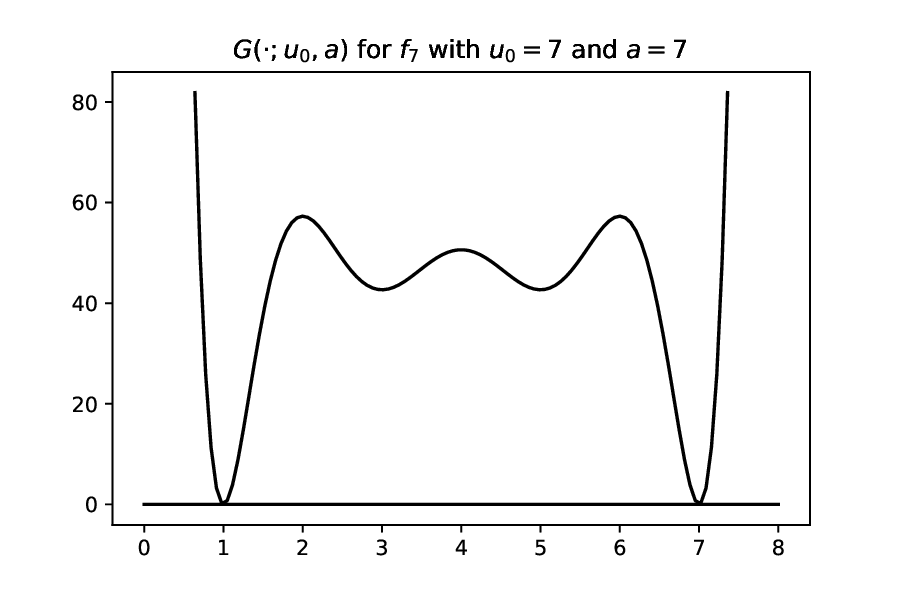}
\end{minipage} 
\caption{For $f_5$ on the left, and $f_7$ on the right, we present the graphs of $G(u;u_0,a)$.
Above, we assume that the entropy is quadratic $U(u) = u^2/2$ and on the left, $u_0 = 5$ and $a=5$,
while on the right $u_0=7$ and $a=7$.}
\label{fig:G}
\end{figure}

\begin{example}
For the flux $f_7$ and the quadratic entropy, the function $G$ can be computed numerically and it's
graph is given in Figure \ref{fig:G} for $u_0=7$ and $a=7$.
\end{example}

\begin{lemma}   \label{lem:exist_tangency}
If Maxwell's equal area rule holds, then there exists at least one tangency point $u^{\natural}$
between $\varphi^{\flat}_0(u_0)$ and $u_0$.
\end{lemma}

\begin{proof}
The function $E(\cdot,u_0)$ vanishes at both $\varphi_0^\flat(u_0)$ and $u_0$. Rolle's theorem gives an interior point where its derivative vanishes. Since this point differs from $u_0$, Proposition~\ref{thm1} identifies it as a tangency point. Its secant speed need not equal the speed $a$ fixed for the phase-plane system.
\end{proof}

In order to characterize non-classical shocks, we will find trajectories to \eqref{Agsystemb}
by identifying invariants regions defined by the sign of $G$. To be precise, we notice that 
if $y \mapsto \big( u(y), v(y)  \big)^T$ is a solution to \eqref{Agsystemb}, then multiplying the equation
for $v_y$ in \eqref{Agsystemb} by $v(y) = u_y(y)$ and integrating from $- \infty$ to $y$, we find
\begin{equation}      \label{pos_constraint2}
        \int_{-\infty}^y \frac{d}{d y'} \bigg(  \frac{v^2(y')}{2} \bigg) \, dy'
          + \alpha \int_{-\infty}^y v^2(y') \, dy'
          = \int_{-\infty}^y \partial_u G\big(  u(y'); u_0,a \big) u_y(y') \, dy'. 
\end{equation}
Applying the boundary conditions \eqref{defcond} and property {\it i)} of
Lemma \ref{propG}, we find
\begin{equation}   \label{pos_constraint}
         \frac{v^2(y)}{2} + \alpha \int_{-\infty}^y v^2(y') \, dy' =  G\big( u(y);  u_0, a \big).
\end{equation}
This identity shows that $G\big( u(y);u_0,a \big)$ must remain positive along a trajectory,
so this motivates our quest to study the sign of $G$.

We begin our analysis by recalling the second property of Lemma \ref{propG} and interpret
it as saying that over the graph of $u \mapsto a(u,u_0)$ in the $(u,a)$-plane, the function $G(u;u_0,a(u,u_0))$
is positive and vanishes at $u=\varphi^\flat_0(u_0)$ by identity \eqref{relation_E_G}. Since the derivative
$\partial_a G(u;u_0,a)$ is negative and independent of $a$, for each $ u\in(\varphi_0^\flat(u_0),u_0)$ there
exists $a^*=a^*(u,u_0) > a(u,u_0)$ at which $G(u;u_0,a^*)=0$. Using the identities  \eqref{relation_E_G}
and \eqref{eq:derivativedaG}, we find that
\begin{align*}
      -\frac{1}{c} E(u,u_0) = & \, G\big( u;u_0,a(u,u_0) \big) 
      = \,  G\big( u;u_0,a(u,u_0) \big) - G( u;u_0,a^* ) \\
      = & \, \int_{a^*}^{a(u,u_0)} \frac{\partial G}{\partial a}(u;u_0,a') \, da' 
      = \, - \frac{(u-u_0)^2}{2} \cdot \big(  a(u,u_0) - a^* \big).
\end{align*}
Solving for $a^*$, we obtain the explicit expression
\begin{equation}     \label{eq:explicitastar}
     a^*(u,u_0) = a(u,u_0) - \frac{2}{c(u-u_0)^2} E(u,u_0).
\end{equation}
For $u = \varphi^\flat_0(u_0)$, the property $E(\varphi^\flat_0(u_0),u_0) = 0$ implies that
$$
      a^* \big( \varphi^\flat_0(u_0), u_0\big) = a \big( \varphi^\flat_0(u_0), u_0 \big).
$$
To show that the expected identity $a^*(u_0,u_0) = a(u_0,u_0) = f'(u_0)$ holds, it suffices to show that
${
      E(u,u_0) = \mathcal{O}\big( (u-u_0)^3 \big),
}$
which we do with a Taylor series expansion 
\begin{align*}
      E(u,u_0) = & \, -a(u,u_0) \cdot \big( U(u)- U(u_0) \big) + \big( F(u)- F(u_0) \big) \\
      = & \, - \Big[ f'(u_0) + \frac{f''(u_0)}{2} (u-u_0) + \mathcal{O}\big( (u-u_0)^2 \big) \Big] \\
          & \, \times \Big[ U'(u_0) (u-u_0) + \frac{U''(u_0)}{2}(u-u_0)^2 + \mathcal{O}\big( (u-u_0)^3 \big)  \Big] \\
         & \, + \Big[ F'(u_0) (u-u_0) + \frac{F''(u_0)}{2}(u-u_0)^2 + \mathcal{O}\big( (u-u_0)^3 \big)  \Big] \\
      = & \, - \Big[ f'(u_0) U'(u_0) (u-u_0) + \frac{f''(u_0)}{2} U'(u_0) (u-u_0)^2 \\
          & \, \phantom{-\Big[} + f'(u_0) \frac{U''(u_0)}{2} (u-u_0)^2 
                +  \mathcal{O}\big( (u-u_0)^3 \big) \Big] \\
        & \,   +  \Big[ F'(u_0) (u-u_0) + \frac{F''(u_0)}{2}(u-u_0)^2 + \mathcal{O}\big( (u-u_0)^3 \big)  \Big].         
\end{align*}
Using the fact that $F'(u) = f'(u) \cdot U'(u)$, we can simplify and deduce
\begin{align*}
      E(u,u_0) = & \, \frac{(u-u_0)^2}{2} \cdot \Big[ -f''(u_0) U'(u_0) - f'(u_0) U''(u_0) + F''(u_0) \Big]
                              +  \mathcal{O}\big( (u-u_0)^3 \big) \\
                     = & \,  \mathcal{O}\big( (u-u_0)^3 \big),
\end{align*}
which when substituted into \eqref{eq:explicitastar} proves $a^*(u_0,u_0)=f'(u_0)$.
Furthermore, we can compute the derivative of $a^*$ with respect to $u$ by starting with the
implicit differentiation of the identity $G\big( u;u_0,a^*(u,u_0) \big) = 0$ to find
$$
        0 = \frac{\partial G}{\partial u} (u;u_0,a^*) + \frac{\partial G}{\partial a} (u;u_0,a^*) \frac{\partial a^*}{\partial u}(u,u_0).
$$
Isolating the derivative and then using the identities \eqref{eq:derivativedaG}, \eqref{eq:derivativeduG},
and \eqref{eq:explicitastar}, we conclude that
\begin{align}
        \frac{\partial a^*}{\partial u}(u, u_0) = & \, \frac{2}{u-u_0}\big[ a(u,u_0) - a^*(u,u_0) \big] \nonumber \\
           = & \,  \frac{4}{c(u-u_0)^3} E(u,u_0) >0.  \label{ineq:dastardu}
\end{align}
This shows that for all $u \in ] \varphi^\flat_0(u_0),u_0[$, $a^*$ is increasing with respect to $u$. 
To justify the regularity also on the diagonal $u=u_0$, we use the integral representation
\be
\label{eq:astar-integral-regularity}
 a^*(u,u_0)=2\int_0^1(1-s)f'\big(u_0+s(u-u_0)\big)\,ds.
\ee
It follows directly from~\eqref{defn:G} by a change of variable and integration in $s$, and is $C^2$ jointly in $(u,u_0)$ because $f\in C^3$.
These conclusions are summarized in the next lemma and Figure \ref{fig:astar}.

\begin{figure}[!htb]
  \psfrag{au}{$a( \, \cdot, u_0 )$}
  \psfrag{astar}{$a^*(\, \cdot \,, u_0)$}
  \psfrag{u}{$u$}
  \psfrag{a}{$a$}
  \psfrag{uequal}{$u=\varphi^{\flat}_0(u_0)$}
  \psfrag{aequal}{$a_0 :=a(\varphi^{\flat}_0(u_0), u_0)$}
  \psfrag{origin}{$(u_0,f'(u_0))$}
\centering \includegraphics[width=0.6\textwidth]{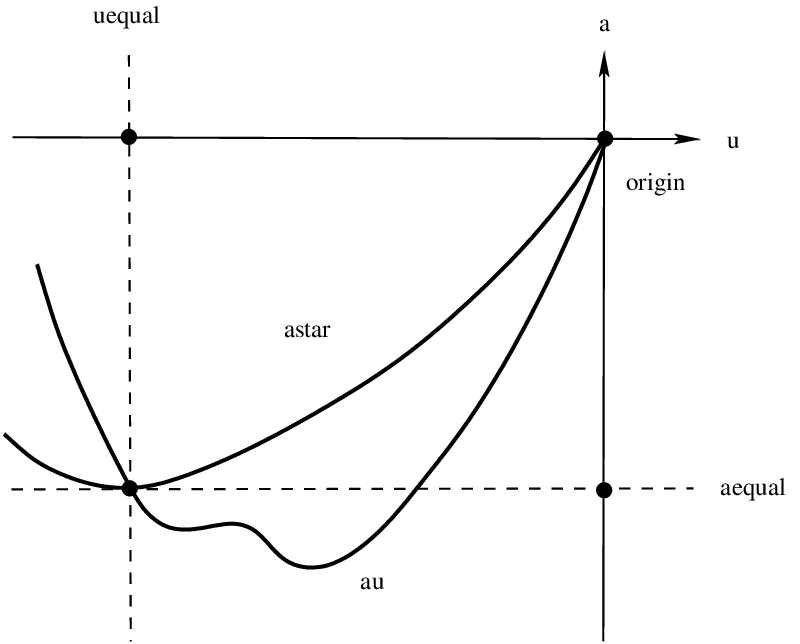}
\caption{Graph of the function $a^*( \cdot,u_0)$ satisfying $G(u;u_0,a^*(u,u_0)) = 0$.}
\label{fig:astar}
\end{figure}

\begin{lemma}  \label{lem:astar}
If Maxwell's equal area rule holds and the entropy is given by \eqref{entropy:explicit}, the function
$a^*: \mathbb{R} \times \mathcal{J}_0 \rightarrow \mathbb{R}$ given by \eqref{eq:explicitastar}
has the following properties.
\begin{itemize}
\item[i)] $a^*$ is $C^2$.
\item[ii)] For each fixed $u_0 \in \mathcal{J}_0$ and $u \in \mathbb{R}$, $G\big( u;u_0,a^*(u, u_0) \big) = 0$. 
\item[iii)] We have the constraints $a^*(u_0,u_0) = f'(u_0) = a(u_0,u_0)$, $a^*\big( \varphi^\flat_0(u_0), u_0 \big) = a(\varphi^\flat_0(u_0),u_0)$ and 
$$
           a^*(u,u_0) > a(u,u_0), \quad \forall u \in ]\varphi^\flat_0(u_0),u_0[.
$$ 
\item[iv)] For each fixed $u_0 \in \mathcal{J}_0$ and $u \in ]\varphi^\flat_0(u_0),u_0[$,
$$
        \frac{\partial a^*}{\partial u}(u,u_0) > 0.
$$
\end{itemize} 
\end{lemma}

One of the consequences of this result is that $a^*(u,u_0)$ is invertible for
every $u$ in $[\varphi^\flat_0(u_0), u_0]$, which we denote as
$$
            u^*(a,u_0) := \big( a^*(\cdot, u_0) \big)^{-1}(a).
$$
In this case, for all $u\in]u^*(a,u_0),u_0[$ we have $G(u;u_0,a)>0$, and $\partial u^*/\partial a(a,u_0)>0$ for $a(\varphi_0^\flat(u_0),u_0)<a<f'(u_0)$. 
While for convex conservation laws the states  connected to $u_0$ are
parametrized by the wave speed, the function $a^*(u,u_0)$ performs
a similar function by delimiting the set of states for which we can look for
travelling waves. As Lemma \ref{lem:astar} has shown, it will be convenient to
introduce some notation for the critical wave speed
$$
      a_0 := a\big( \varphi^\flat_0(u_0), u_0 \big).
$$

In the remaining saddle constructions, the prescribed assumptions on the speed $a$ are
\be
\label{hyp:speed-assumptions}
     a_0<a<f'(u_0),\qquad f'(u_i)\neq a\quad\text{for every equilibrium }u_i.
\ee
Endpoints of this interval are treated only when specified explicitly.

The prescribed speed assumptions~\eqref{hyp:speed-assumptions} include the restriction $a>a_0$, beyond hyperbolicity. For $f(u)=u^3$ and $u_0=1$, the explicit profile $u(y)=1/5-(4/5)\tanh(4y/(5\sqrt2))$ has $a=19/25$ and $\alpha=3\sqrt2/5$, whereas $a_0=1$. This undercompressive profile is outside the speed range classified here.

We make a final observation about the behavior of $G$ at $u^*$ because it
could be tempting to assume that it is always an equilibrium. For $a_0<a<f'(u_0)$, we have that 
$$
   0 = G(u^*(a,u_0);u_0,a).
$$  
Taking the implicit derivative with respect to $a$, and using \eqref{ineq:dastardu},
we find
\begin{align*}
    0 = & \, \partial_u G(u^*;u_0,a) \frac{\partial u^*}{\partial a}(a,u_0) + \partial_a G(u^*;u_0,a) \\
    = & \, \partial_u G(u^*;u_0,a ) \frac{c( u^*-u_0)^3}{4 E( u^*,u_0)} - \frac{( u^*-u_0)^2}{2}.
\end{align*}
This leads to the identity
\begin{equation}    \label{eq:explicitdG}
     \partial_u G(u^*;u_0,a) = \frac{2 E(u^*,u_0)}{c( u^* -u_0)}.
\end{equation}
In particular, under the prescribed speed assumptions~\eqref{hyp:speed-assumptions} $u^*\neq\varphi^\flat_0(u_0)$
and $\partial_u G(u^*;u_0,a)>0$, $u^*$ is not an equilibrium point.


\subsection{Properties of diffusive-dispersive trajectories}
\label{sec:5}

We now present a series of fundamental results that will underpin the proof of our main result.
These results describe some of the elementary properties of solutions to
the system \eqref{Agsystemb}, first with the asymptotic properties of these solutions.
We then focus on the zero-diffusion solutions which bound absorbing regions of
the system. This is then followed with a monotonicity lemma and continuity results
on the variation of these solutions with respect to the diffusion-dispersion ratio.


\begin{lemma}   \label{lem:convergence2saddle}
Suppose that Maxwell's equal area rule holds, the entropy is given by \eqref{entropy:explicit}, $\alpha>0$, and $a_0\leq a<f'(u_0)$. The nonconstant solution on the left outgoing branch $C^3_{0,\alpha}$ is defined for all $y\in\RR$, remains in $u^*(a,u_0)<u(y)<u_0$, and converges as $y\to+\infty$ to an equilibrium $(u_k,0)$ strictly inside this interval. More generally, every forward bounded solution of \eqref{Agsystemb} converges to an equilibrium if the bounded region contains finitely many equilibria.
\end{lemma}

\begin{proof}
Write $u^*=u^*(a,u_0)$ and suppress the fixed arguments of $G$. The branch initially has $u<u_0$ and $v<0$. At a hypothetical first finite crossing of $u^*$ or $u_0$, the equation~\eqref{pos_constraint} would give
\be
\label{eq:boundary-crossing-energy}
 0\leq \frac{v(y)^2}{2}+\alpha\int_{-\infty}^y v(s)^2\,ds=0,
\ee
contradicting the nonconstant trajectory. Thus $u^*<u(y)<u_0$ throughout its interval of existence. The same identity gives $v^2/2\leq\max_{[u^*,u_0]}G$, so the solution stays in a compact rectangle. The vector field is locally Lipschitz there, and the solution therefore extends for all positive times.

For the more general convergence assertion, we use the energy identity~\eqref{eq:lyapunov-correction}.
On a bounded trajectory $H$ is bounded below, hence $v\in L^2([0,+\infty))$. The equation for $v_y$ gives a uniform bound on $v_y$. If $v$ did not tend to zero, this bound would produce disjoint intervals of a fixed positive length on which $|v|$ is bounded below, contradicting its square integrability. Thus $v(y)\to0$.

Let $u(y_j)\to\bar u$ along a sequence $y_j\to\infty$. Continuous dependence on bounded time intervals shows that the solution starting at $(\bar u,0)$ is a limit of the translated solutions. Its second component vanishes identically, since $v(y)\to0$, so $\partial_uG(\bar u)=0$. The set of limit points of the continuous scalar function $u$ is an interval: if two distinct limits occur, the intermediate value theorem supplies every value between them along times tending to infinity. Since there are finitely many equilibria, this interval is a singleton. This proves convergence.

Finally, on the branch from $u_0$, the integral in \eqref{pos_constraint} has a strictly positive limit. Consequently $G(u_k)>0$, which excludes both $u^*$ and $u_0$. This completes the dissipative argument; the case $\alpha=0$ is described by Lemma~\ref{lem:bubble}.
\end{proof}

Next, given $u_0$ and  the speed $a$ satisfying the prescribed speed assumptions~\eqref{hyp:speed-assumptions}, we know that the largest root of
$G(\cdot\,;u_0,a)$ strictly below $u_0$ occurs at $u=u^*(a,u_0)$, so for the remainder of the analysis we will ignore all roots less than
$u^*(a,u_0)$ and redefine $K$ so that we have $K+ 2$ equilibrium points ordered as
$$
      u^*(a,u_0) < u_{K+1} < u_K < \cdots < u_1 < u_0.
$$ 
Since $G$ vanishes at $u^*$, $\partial_u G\big( u^*(a,u_0) ; u_0, a \big) \geq 0$
so that according to Lemma \ref{propG}, $u_{K+1}$ must be a stable equilibrium point. By renaming $K$, we have
in effect made the convenient decision that the saddle equilibria will be $u_K < u_{K-2} < \cdots < u_2< u_0$. 
The next result is an analogue of Proposition 3.1 from \cite{BedChaCoqLeF}.
 
%
%
%
\begin{lemma}
Suppose that Maxwell's equal area rule holds and the entropy is given by \eqref{entropy:explicit}.
For  $\alpha>0$ and $a_0\leq a<f'(u_0)$, a traveling wave connecting $(u_0,0)$ to a saddle equilibrium $(u_k,0)$ with $u_k<u_0$ must leave $(u_0,0)$ in the quadrant $Q^3_0$.
\end{lemma}

\begin{proof}
A trajectory leaving in $Q^1_0$ initially has $u>u_0$. To converge to $u_k<u_0$, it must cross $u=u_0$ at a finite time $y_0$. 
The equation~\eqref{pos_constraint} then gives $v(y_0)^2/2+\alpha\int_{-\infty}^{y_0}v^2=G(u_0)=0$, contradicting a nonconstant trajectory.
\end{proof}

So far, we have shown that the left outgoing trajectories, for $\alpha>0$, must eventually approach one of the
remaining $K+1$ equilibrium points strictly between $u^*(a,u_0)$ and $u_0$. Our objective is now to show that 
as a function of $\alpha$ the set of trajectories
form something like a foliation of a bounded region of phase-space, here interpreting foliation roughly 
in the sense of differential geometry. We will demonstrate that the purely dispersive trajectories
define the boundary of this region and use these boundaries to identify
critical diffusion-dispersion ratios $\alpha_n(k)$ for which there exists
travelling waves connecting $u_0$ to $u_k$ after $n$ crossings of the $u$-axis. 
In contrast to the earlier paper \cite{BedChaCoqLeF}, we will see that the
collection of critical diffusion-dispersion ratios cannot be counted in a manner as
simple as when there are only two saddle equilibria.

\begin{figure}[!htb] 
\centering
\begin{minipage}{.5\textwidth}
       \centering  \includegraphics[width=\linewidth]{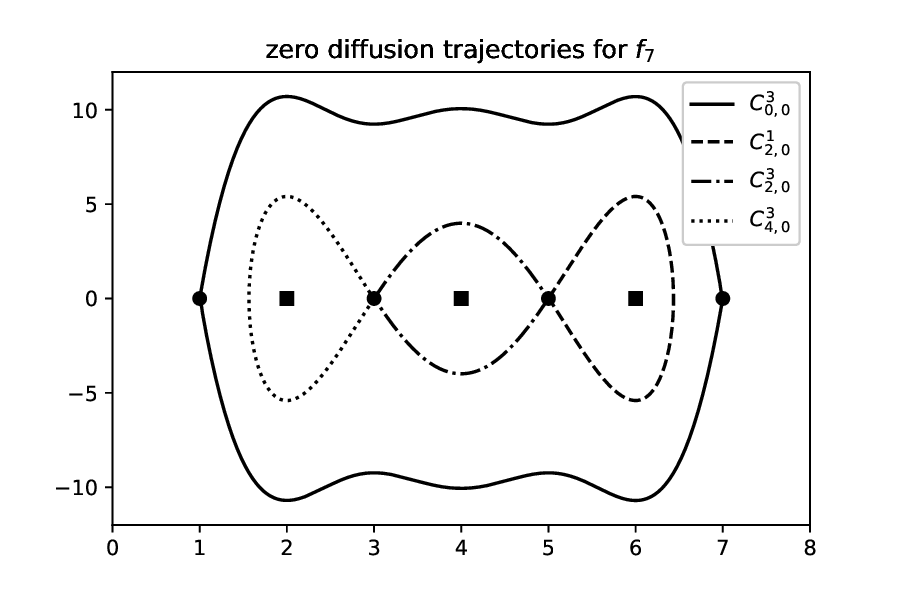}
\end{minipage}%
\begin{minipage}{.5\textwidth}
       \centering  \includegraphics[width=\linewidth]{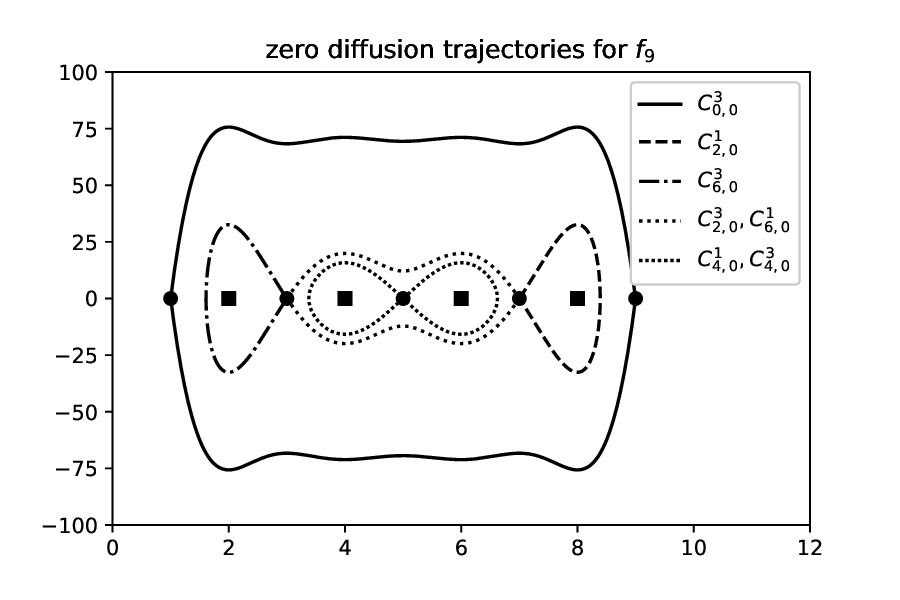}
\end{minipage} 
\caption{Illustration of zero diffusion trajectories for the fluxes $f_k$ when $u_0=k$ and $a=k$. The circular dots are saddle 
equilibrium points and the square dots are centers for zero diffusion (stable equilibria for positive diffusion). }
\label{fig:bubble}
\end{figure}

The polynomial figures at $(u_0,a)=(k,k)$ illustrate Hamiltonian trajectories, without imposing the prescribed speed assumptions~\eqref{hyp:speed-assumptions} of the later theorems. For $f_7$, the largest Maxwell zero is approximately $2.52695$ and gives $a_0\approx11.81574>7$. Thus the plot at $a=7$ is not an instance of that hypothesis. The numerical spline examples below are covered separately by the local fixed-speed assumptions.

\begin{figure}[!htb]
\centering \includegraphics[width=0.6\textwidth]{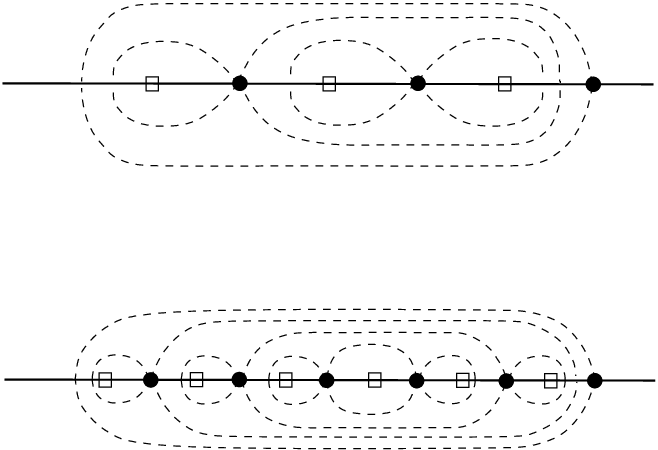}
\caption{Graph of the zero diffusion trajectories in two different situations. Notice
the nestings of the bubbles. The dark circles are saddle equilibria and the empty
squares are stable equilibria.}
\label{fig:configuration1}
\end{figure}

%

The phase-plane results from Lemma~\ref{lem:convergence2saddle} onward also hold for a flux of class $C^2$ on a neighborhood of the compact interval $[u^*,u_0]$, provided the equilibria there are finite and simple, $u_0$ is a saddle, $G(u^*)=G(u_0)=0$, $G>0$ inside, and $G'(u^*)>0$. Indeed, the vector field is $C^2$, the saddle curves and their first parameter variations exist, and all energy and comparison arguments use at most two derivatives of $f$. The $C^2$ assertion for the two-variable function $a^*$ in Lemma~\ref{lem:astar} is not needed for these fixed-speed results. This local formulation applies directly to the cubic splines below, without presuming that an arbitrary smoothing preserves saddle connections.

We begin with an important generalization of Proposition 2.3 from Bedjaoui et al. \cite{BedChaCoqLeF}.

\begin{lemma}[bubble lemma]   \label{lem:bubble}
Assume Maxwell's equal area rule, the entropy \eqref{entropy:explicit}, $\alpha=0$, and the speed $a$ satisfying the 
prescribed speed assumptions~\eqref{hyp:speed-assumptions}. Then
\be
\label{eq:outsidebubble}
 C^3_{0,0}=C^2_{0,0},
\ee
with its unique regular turning point at $u^*(a,u_0)$. If $K\geq2$, the extreme internal saddles satisfy
\be
\label{eq:extreme-bubbles}
 C^3_{K,0}=C^2_{K,0},\qquad C^1_{2,0}=C^4_{2,0}.
\ee
For an internal saddle $u_k$, let $u^-_k<u_k$ and $u^+_k>u_k$ be the first points, to its left and right respectively, at which $G(u)=G(u_k)$. Each exists in $[u^*,u_0]$ and satisfies $G(u)>G(u_k)$ between that point and $u_k$.
\begin{itemize}
\item If $u^-_k$ is not an equilibrium, then $C^3_{k,0}=C^2_{k,0}$ is homoclinic, with turning point $u^-_k<u_{k+1}$. If $u^-_k=u_j$ is an equilibrium, it is a saddle and
\be
\label{eq:first-contact-connections}
 C^3_{k,0}=C^4_{j,0},\qquad C^1_{j,0}=C^2_{k,0}.
\ee
\item If $u^+_k$ is not an equilibrium, then $C^1_{k,0}=C^4_{k,0}$ is homoclinic, with turning point $u^+_k>u_{k-1}$. If $u^+_k=u_j$ is an equilibrium, it is a saddle and the corresponding two reflected connections join $u_k$ to $u_j$.
\end{itemize}
In particular, if $G(u_k)$ differs from every other saddle height, both branches form homoclinic loops.

At $a=a_0$, suppose the equilibria remain simple and write $\varphi=\varphi_0^\flat(u_0)=u^*(a_0,u_0)$. With quadrants centered at $\varphi$, the outer boundary is instead
\be
\label{eq:outsidebubblespecial}
 C^3_{0,0}=C^4_{\varphi,0},\qquad C^1_{\varphi,0}=C^2_{0,0}.
\ee
\end{lemma}

\begin{proof}
Along a branch issuing from a saddle $u_k$, conservation of $H$ defined in \eqref{eq:lyapunov-correction} gives
\be
\label{eq:bubble.eq2}
 \frac{v^2}{2}=G(u)-G(u_k).
\ee
Since $G''(u_k)>0$, this difference is positive on both sides near $u_k$. For $k>0$, $G(u_k)>0$ while $G(u^*)=G(u_0)=0$, so a first level contact exists on each side. A regular contact is reached in finite time because the integral of $1/\sqrt{2(G(u)-G(u_k))}$ is finite at a simple zero. The symmetry $(u(y),v(y))\mapsto(u(-y),-v(-y))$ then gives the homoclinic loop. An equilibrium contact cannot be a local maximum of $G$, since the level difference is positive on the approaching side. It is therefore a saddle. There $G(u)-G(u_j)=G''(u_j)(u-u_j)^2/2+o((u-u_j)^2)$, so the time integral diverges and the branch approaches the saddle asymptotically. Reflection supplies the second connection.

The adjacent stable equilibria $u_{k+1}$ and $u_{k-1}$ are strict maxima of $G$, giving the stated inequalities for regular turning points. At the extreme internal saddles there is no intervening saddle before the outer boundary, which proves the loops at $u_K$ and $u_2$.

For $k=0$ the same identity reads $v^2/2=G(u)$, and $G>0$ on $(u^*,u_0)$. Under the prescribed speed assumptions~\eqref{hyp:speed-assumptions}, $G'(u^*)>0$, so the first contact is regular and reflection proves \eqref{eq:outsidebubble}. At $a=a_0$ the first contact is the simple saddle $\varphi$, giving \eqref{eq:outsidebubblespecial} by the preceding asymptotic argument.
\end{proof}

Under the prescribed speed assumptions~\eqref{hyp:speed-assumptions}, the outer loop bounds the region used below. Internal saddle connections additionally require $G(u)>G(u_k)$ at every intervening state; equal endpoint heights alone do not imply a connection.

The zero diffusion trajectories, which we refer to as {\it bubbles}, are all solutions to the 
same system \eqref{Agsystemb}, hence they cannot cross, and
must subdivide the region in $C^3_{0,0}$ into a collection of nested
sets containing the equilibria.  
It is easy to verify directly, but will follow 
from the more general Lemma \ref{lem:monotonicity}, that the regions bounded by 
the bubbles are absorbing regions for the differential equation. 
The relative positions and nestings of the bubbles will determine the paths
taken by solutions to the differential equations  \eqref{Agsystemb} as they approach an equilibrium point.
We will eventually see that these bubbles are organized into {\it bowties},
which we define below, and that the overall structure of the bubbles is
characterized by the relative nesting of these bowties.

\begin{definition}
We say that two saddle equilibrium $u_j < u_k$ are connected if
$$
      C^3_{k,0} = C^4_{j,0} \quad \text{and} \quad C^1_{j,0} = C^2_{k,0}.
$$
A bowtie is a sequence of $m$ connected saddle equilibria
defined by an increasing function $k:\{1, \ldots,m\} \rightarrow 2 \mathbb{Z}$ from which we 
determine the sequence
$$
    u_{k(m)} < u_{k(m-1)} < \cdots  < u_{k(1)}.
$$
We say that a bowtie is elementary if the interior each of its bubbles contains a 
single stable equilibrium point. 
\end{definition}

In Figure \ref{fig:bowtie}, we illustrate three different bowties to give the reader some intuition.

\begin{figure}[!htb]
\psfrag{R}{$R$}
\psfrag{D}{$D$}
\psfrag{L}{$L$}
\psfrag{RL}{$RL$}
\psfrag{RRL}{$R\underline{RL}$}
\psfrag{RLL}{$\underline{RL}L$}
\psfrag{RDL}{$RDL$}
\psfrag{RRDL}{$R\underline{RDL}$}
\psfrag{RRDRL}{$RRD\underline{RL}$}
\psfrag{RRDDL}{$R\underline{RDDL}$}
\psfrag{RRDLL}{$R\underline{RDL}L$}
\psfrag{RRDRLL}{$RRD\underline{RL}L$}
\centering \includegraphics[width=1.0\textwidth]{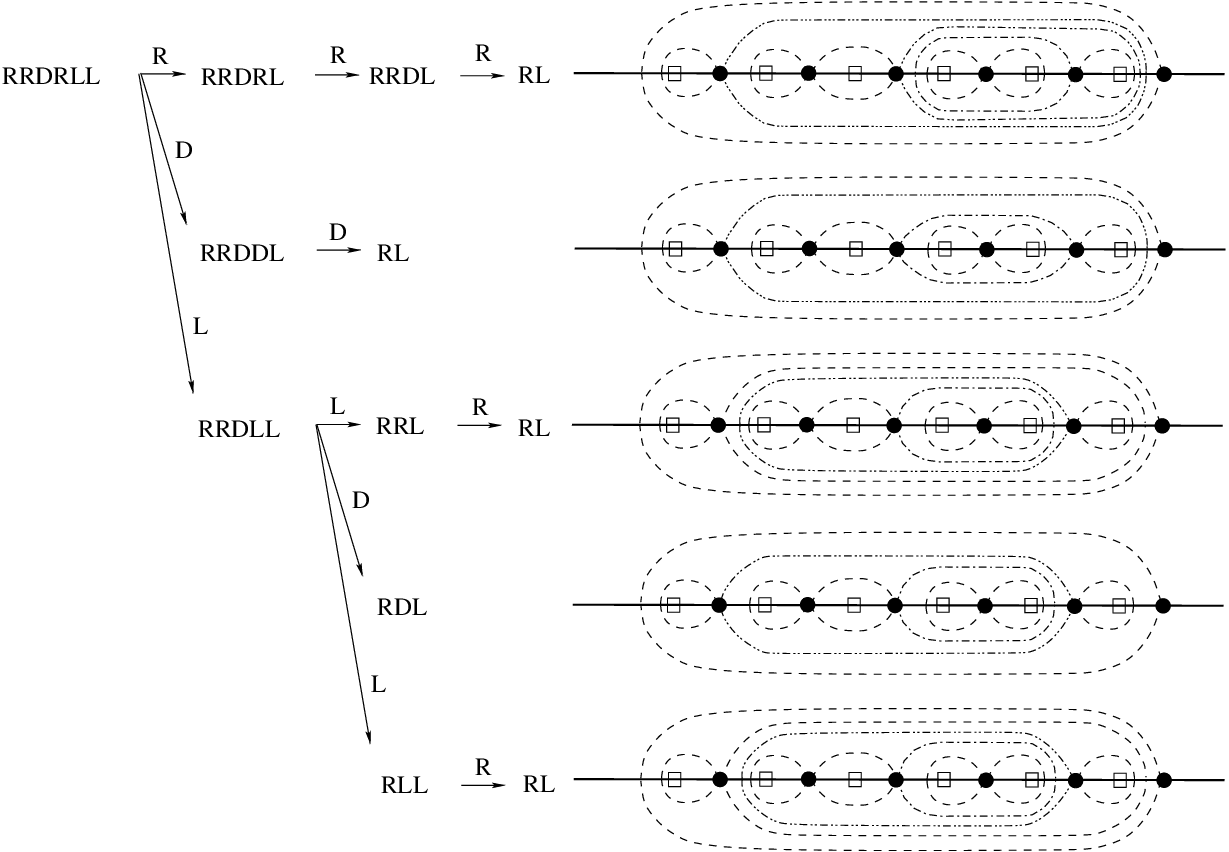}
\caption{This figure illustrates how the reduction rule is applied to an initial configuration $RRDRLL$
to generate all possible nestings of bubbles. Every time a bowtie is identified, it is underlined
and a decision is made to replace it by either a bubble $R, D$ or $L$. The decision is indicated
above each arrow. We obtain $5$ possible configurations. The first bubble that is decided is indicated as a dash-dot line, 
the second decision was described as a dash-dot-dot line, and the third was a dash-dot-dot-dot line.}
\label{fig:bowtie}
\end{figure}

In the nesting comparisons below, $C^i_{j,0}\subset C^\ell_{k,0}$ compares the bounded regions 
delimited by the respective separatrices and their reflections across the $u$-axis.

The symbolic description below concerns $K\geq2$; for $K=0$ the outer region contains only its stable equilibrium.

The next step is to characterize all possible nestings of bubbles since
ultimately these options will determine the possible paths taken by travelling waves. For each interval
$]u_{k+2},u_k[$ bounded by two saddle equilibria, the trajectories
$$
         C^1_{k+2,0} \quad \text{and} \quad C^3_{k,0} 
$$
cannot cross but $C^1_{k+2,0}$ is to the right of $u_{k+2}$
and has its first level contact at $u^+_{k+2}$ to the right of $u_{k+1}$ while $C^3_{k,0}$
is to the left of $u_k$ and has its first level contact at $u^-_k$ to the left of $u_{k+1}$, and so only one
these three options can occur.
\begin{itemize}
\item[1)] $C^1_{k+2,0} \subset C^3_{k,0}$
\item[2)] $C^3_{k,0} \subset C^1_{k+2,0}$
\item[3)] $C^1_{k+2,0} = C^2_{k,0}$ and $C^3_{k,0}  = C^4_{k+2,0}$
\end{itemize} 
The interval $]u_{k+2},u_k[$ will be said to be either left-constrained, right-constrained, or
dual-constrained for the respective configurations described above. We briefly remark
that it should be clear that dual-constrained intervals are simply limits of the first two
cases as the bubbles approach each other.
Given any generic flux with $K+1$ equilibrium points strictly between $u_0$ and $u^*$, the
largest root of $G$ smaller than $u_0$, then each interval
$$
       ]u_K, u_{K-2}[, \ldots,]u_4, u_2 [ 
$$
is said to be either right, left, or dual constrained, but the interval $]u_2,u_0[$
must always be left-constrained while the interval $]u^*,u_K[$ must always be right-constrained
since all bubbles are contained inside the larger $C^3_{0,0}$.
Let us summarize this information as a sequence of symbols
\begin{equation}   \label{eq:initialconstraints}
           R P_{K-1} P_{K-3} \cdots P_3 L
\end{equation}
where $P_j \in \{ R,D,L \}$ indicates if the interval $]u_{j+1},u_{j-1}[$ between two saddle
equilibria is right, left, or dual constrained. From this information, we will introduce a
series of choices that completely determine the nesting of the bubbles.
If we have a subset of the sequence \eqref{eq:initialconstraints} formed by the symbols $R$ and $L$ 
bounding a continuous sequence of $D$'s, like
$$
      RL, RDL, RDDL, RDDDL,
$$
then that sequence represents a bowtie. If $R$ belongs to the interval $]u_{j+2},u_j[$ 
and $L$ to the interval $]u_{k+2},u_k[$,
then $C^1_{j+2,0}$ must surround the bowties and so does $C^3_{k,0}$, unless
$$
     C^1_{j+2,0} = C^2_{k,0} \quad \text{and} \quad C^3_{k,0} = C^4_{j+2,0}.
$$  
So, we can replace the bowtie sequence by either a single letter $L, R$ or $D$,
indicating whether we have
$$
    C^1_{j+2,0} \subset C^3_{k,0}, \quad  C^3_{k,0} \subset C^1_{j+2,0}, 
$$
or
$$
      C^1_{j+2,0} = C^2_{k,0} \quad \text{and} \quad C^3_{k,0} = C^4_{j+2,0}.
$$
We can say that the interval $]u_{J+1}, u_{K+1}[ $ containing the bowtie is either left-constrained,
right-constrained, or dual-constrained.
This rule can be applied repeatedly every time a bowtie is identified in the sequence,
and bowties must always be present because, even as this rule is applied repeatedly, the sequence always begins with an $R$
on the left and an $L$ on the right. Applying these operations produces a sequence
of choices that entirely describe the nesting of the bubbles and eventually reduce the sequence to 
a single bowtie representing the largest bowtie in $C^3_{0,0}$. To illustrate this reduction process, 
we refer the reader to Figure \ref{fig:bowtie}
where this process is applied to the initial sequence $RRDRLL$. We observe in that example
that there are five different possible nestings of 
the bubbles that are consistent with this initial description \eqref{eq:initialconstraints} of the geometry of each interval.

At this point, it is useful to take a moment to introduce a rewriting of the system of differential equations \eqref{Agsystemb}. 
Assuming that we have a solution $y \mapsto (u(y),v(y))^T$ to the system and that for 
some $y^*$, $v(y^*) = u_y(y^*) \neq 0$, then $u(y)$ can be inverted and we can write $v(u) := v \big( y(u)\big)$ with 
\begin{align*} 
         \frac{d v}{dy}\big( u(y) \big) = & \, \frac{d v}{d u}\big( u(y) \big) \cdot \frac{d u}{d y}\big( y \big) = \frac{d v}{d u}(u) \cdot v(u).
\end{align*}
The system \eqref{Agsystemb} becomes 
\begin{equation}     \label{eq:systemwrtu}
           \frac{d v}{d u}(u) \cdot v(u) = - \alpha v(u) + \partial_u G(u;u_0,a),
\end{equation}
at least whenever $v(u) \neq 0$. Usually, the left-hand side will be rewritten as
\begin{equation}     \label{eq:systemwrtu2}
             \frac{d}{d u}\bigg(  \frac{v(u)^2}{2} \bigg) = - \alpha v(u) + \partial_u G(u;u_0,a),
\end{equation}
which, upon integration in $u$, leads to identities equivalent to \eqref{pos_constraint}
if we made the change of variable $du=v dy$. This equation
will be used in the proof of the next result, but also a few times later in the text.

\begin{figure}[!htb]
  \psfrag{a1la2}{$0 \leq  \alpha_1 < \alpha_2$}
  \psfrag{uk}{$(u_k, 0)$}
  \psfrag{u}{$u$}
  \psfrag{v}{$v$}
  \psfrag{ukp2}{$u_{k+2}$}
  \psfrag{ukp1}{$u_{k+1}$}
  \psfrag{ukp0}{$u_k$}
  \psfrag{v1}{$v_1$}
  \psfrag{v2}{$v_2$}
  \psfrag{C1a1}{$C^1_{k,\alpha_1}$}
  \psfrag{C1a2}{$C^1_{k,\alpha_2}$}
  \psfrag{C2a1}{$C^2_{k,\alpha_1}$}
  \psfrag{C2a2}{$C^2_{k,\alpha_2}$}
  \psfrag{C3a1}{$C^3_{k,\alpha_1}$}
  \psfrag{C3a2}{$C^3_{k,\alpha_2}$}
  \psfrag{C4a1}{$C^4_{k,\alpha_1}$}
  \psfrag{C4a2}{$C^4_{k,\alpha_2}$}
\centering
\begin{minipage}{.5\textwidth}
       \centering  \includegraphics[width=\linewidth]{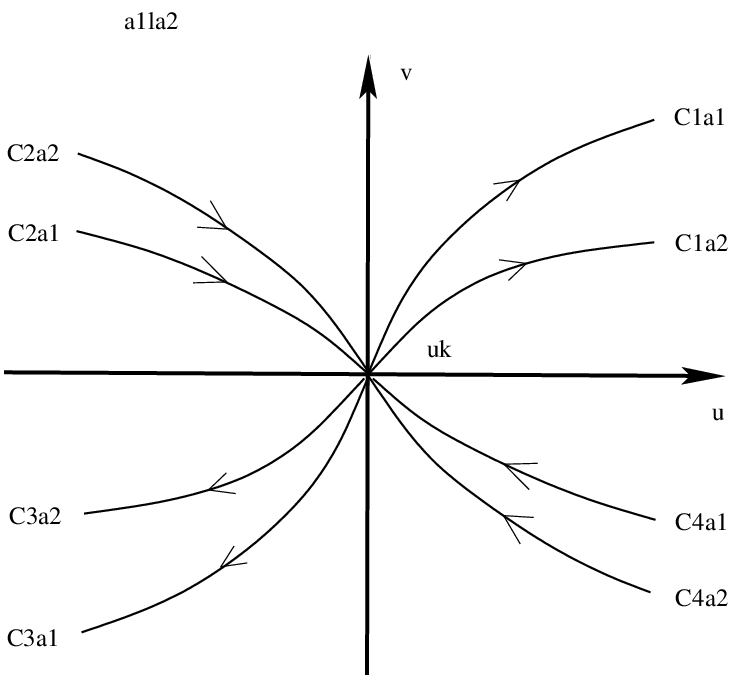}
\end{minipage}%
\begin{minipage}{.5\textwidth}
       \centering  \includegraphics[width=\linewidth]{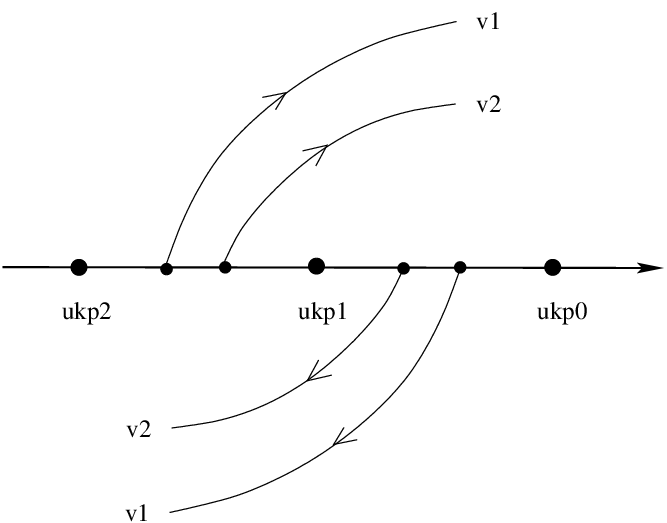}
\end{minipage} 
\caption{The illustration on the left concerns the cases (i) and (ii) of the monotonicity lemma, Lemma \ref{lem:monotonicity},
while the one on the right describes case (iii).}
\label{fig:monotone}
\end{figure}

%
\begin{lemma}[monotonicity lemma]    \label{lem:monotonicity}
Assume that Maxwell's equal area rule holds, that the entropy is given by \eqref{entropy:explicit}, and that
the speed $a$ satisfies the prescribed speed assumptions~\eqref{hyp:speed-assumptions}. For a pair of trajectories corresponding to
diffusion-dispersion ratios $0 \leq \alpha_1 < \alpha_2$ and leaving a saddle equilibrium $u_k$,
we have the following properties.
\begin{itemize}
\item[(i)] For $i=1,3$, as long as $C^i_{k,\alpha_1}$ and $C^i_{k,\alpha_2}$ are above (respectively below) the $u$-axis, then
              $C^i_{k,\alpha_1}$ remains above (respectively below) $C^i_{k,\alpha_2}$.
\item[(ii)]  For $i=2,4$, as long as $C^i_{k,\alpha_1}$ and $C^i_{k,\alpha_2}$ are above (respectively below) the $u$-axis, then
              $C^i_{k,\alpha_2}$ remains above (respectively below) $C^i_{k,\alpha_1}$.               
\end{itemize} 
More generally, consider two solutions $U_i(y) = \big( u_i(y), v_i(y) \big)^T$, $i=1,2,$
associated to $\alpha_i$ satisfying $\alpha_1 < \alpha_2$ and which cross the $u$-axis
at $y_i$, that is $v_i(y_i) = 0$, $i=1,2$.
\begin{itemize}
\item[(iii)] If $u_k < u_1(y_1) < u_2(y_2) < u_{k-1}$ ($u_{k+1} < u_2(y_2) < u_1(y_1) < u_k$) for $k$ even, then 
                as long as both solutions remain above (below) the 
                $u$-axis, the graph of $U_2$ is bounded above (below) by the graph of $U_1$.
\end{itemize}
In (iii), the same comparison holds whenever the turning points lie in a common interval where $G'$ has the indicated strict sign. In particular, it applies on the left outer interval $(u^*,u_{K+1})$, where $G'>0$.
\end{lemma}

\begin{proof}
We begin by studying the first two cases which are distinct from the last one since the trajectories
begin at a saddle equilibrium. In the first case when $i=1$, the solutions beginning at $u_k$
asymptotically leave in the direction of the eigenvector
$$
          \overline{r}_k = (1, \, \overline{\mu}_k)^T
$$
and we verified in Lemma \ref{valpropres} that 
$$
    \frac{\partial \overline{\mu}_k}{\partial \alpha} < 0.
$$
Hence, when $i=1$ there is a small interval $]u_k,u_k+\epsilon[$ over which $C^1_{k,\alpha_1}$ is
above $C^1_{k,\alpha_2}$ and in particular there is a $y_1, y_2$ for which 
$$
    u_1(y_1)=u_2(y_2) \quad \text{but} \quad v_1(y_1)> v_2(y_2) > 0.
$$
We now observe that since $v_j(y_j) = d u_j/dy(y_j) \neq 0$ for $j=1,2$, then
$u_j$ is invertible with respect to $y$ and the trajectories $v_j(y)$ are solutions to \eqref{eq:systemwrtu}
\begin{align}
           \frac{d v_1}{du}(u) \cdot v_1(u) = & \, - \alpha_1 v_1(u) + \partial_u G(u;u_0,a),  \label{eq:mono1} \\
           \frac{d v_2}{du}(u) \cdot v_2(u) = & \, - \alpha_2 v_2(u) + \partial_u G(u;u_0,a).  \label{eq:mono2}
\end{align}
To proceed by contradiction, we assume that there exists $u^*>u_k+\epsilon$ at which
$$
     v_1(u^*) = v_2(u^*) > 0.
$$
After subtracting \eqref{eq:mono2} from \eqref{eq:mono1}, we find
$$
     v_1(u^*) \bigg(   \frac{d v_1}{du}(u^*) -   \frac{d v_2}{du}(u^*)   \bigg) = - v_1(u^*) \big( \alpha_1 - \alpha_2 \big).
$$
Since $\alpha_1 < \alpha_2$ and $v_1(u^*) > 0$, we conclude that
\begin{equation}   \label{ineq:mono}
          \frac{d v_1}{du}(u^*) >  \frac{d v_2}{du}(u^*),
\end{equation}
which contradicts our assumption that $C^1_{k,\alpha_2}$ crossed $C^1_{k,\alpha_1}$ for the first time at $u=u^*$.

In the case (i) when $i=3$, then the analysis of the eigenvectors in Lemma \ref{valpropres} again implies that
there exists a small interval $[u_k-\epsilon,u_k[$ over which $C^3_{k,\alpha_1}$
is below $C^3_{k,\alpha_2}$, both below the $u$-axis. As before, as long the curves remain below the
$u$-axis then $v_j$ satisfies a differential equation with respect to $u$, namely \eqref{eq:mono1}
and \eqref{eq:mono2}.  If the curves $C^3_{k,\alpha_1}$ and $C^3_{k,\alpha_2}$ cross while still below the $u$-axis, 
then there is a $u^* < u_k - \epsilon $ at which
$$
     v_1(u^*) = v_2(u^*) < 0.
$$ 
The same argument yields the inequality \eqref{ineq:mono}.
At the crossing $(u^*,v_1(u^*)) = (u^*,v_2(u^*))$ since $v_1(u) < v_2(u)$ to the right of $u^*$,
then we would need to have
$$
         \frac{d v_1}{du}(u^*) \leq \frac{d v_2}{du}(u^*),
$$
which would contradict \eqref{ineq:mono}. Therefore the curves cannot cross
while still below the $u$-axis.

We now briefly consider the case (ii) and focus on $i=2$. In this case, the eigenvectors are
$
     \underline{r}_k = (1, \, \underline{\mu}_k)^T
$
and we also have
\be
     \frac{\partial \underline{\mu}_k}{\partial \alpha} < 0,
\ee
yet now $\underline{\mu}_k<0$ at the saddle equilibrium. Therefore, there is a small interval $[u_k- \epsilon,u_k[$
over which $C^2_{k,\alpha_1}$ is below $C^2_{k,\alpha_2}$. The same argument as before states that if they cross,
then geometrically \eqref{ineq:mono} holds. Yet, if they cross then geometrically they must satisfy
$$
            \frac{d v_1}{du}(u^*) \leq \frac{d v_2}{du}(u^*) ,
$$
must hold thereby contradicting \eqref{ineq:mono}. This proves the case (ii) when $i=2$ while if
$i=4$, we still have that $C^4_{k,\alpha_2}$ is below $C^4_{k,\alpha_1}$ over some small interval $]u_k,u_k+\epsilon]$,
both below the $u$-axis. If they cross at $u^*> u_k+\epsilon$, then \eqref{ineq:mono} again holds but the
geometry implies that
$$
               \frac{d v_1}{du}(u^*) \leq \frac{d v_2}{du}(u^*).
$$
We again conclude that the two trajectories cannot cross as long as they are both below the $u$-axis.

The proof of the case (iii) is almost the same as the one presented earlier.  We assume that
$u_k < u_1(y_1)< u_2(y_2) < u_{k-1}$ where $k$ is even and $u_k$ is a saddle equilibrium. Lemma
\eqref{propG} states that $\partial_u G(u;u_0,a)>0$ for all $u \in] u_k, u_{k-1}[$ so that
$$
                 \frac{d v_i}{d y }(y_i) = \partial_u G \big( u_i(y_i); u_0,a \big) 
$$
and $u_i$ decreases to $u_i(y)$ for $y< y_i$ before increasing for $y>y_i$.
The solution $u_i$ will continue to increase until it reaches $u_1(y_1^*) = u_2(y_2)$,
at which point $v_1(y_1^*) > v_2(y_2)=0$,  and in particular for $\bar{u}$ slightly
larger than $u_1(y_1^*) = u_2(y_2)$, we have
$$
         0< v_2(\bar{u}) < v_1(\bar{u}).
$$         
If the trajectories cross at $u^*> \bar{u}$, then the inequality \eqref{ineq:mono}
holds but the geometry implies that the 
crossing must satisfy
$$
               \frac{d v_1}{du}(u^*) \leq \frac{d v_2}{du}(u^*).
$$
So a crossing is not possible to the right of $\overline{u}$. 
A similar argument can be used when the trajectories cross the $u$-axis between $]u_{k+1},u_k[$
for $k$ even.
\end{proof}

Using the monotonicity lemma, we can conclude that the regions bounded
by the zero diffusion trajectories, the so-called bubbles, define absorbing regions.

\begin{corollary}
Consider a solution $U(y)$ to the system \eqref{Agsystemb} with $\alpha>0$ that crosses at $y=y_0$ a zero diffusion trajectory bounding a bubble $A$. Then for all $y>y_0$ the solution remains in $A$ and converges to an equilibrium in its interior.
\end{corollary}

\begin{proof}
The boundary has a constant energy $H=h$, and the interior has $H<h$. Indeed, its upper and lower arcs are $v=\pm\sqrt{2(G(u)+h)}$. At a non-equilibrium point of an arc with $v\neq0$, the equation~\eqref{eq:lyapunov-correction} gives $H_y<0$, so the crossing is inward. At a regular turning point, $v_y=G'(u)\neq0$ and $H$ also decreases strictly immediately afterwards. An equilibrium cannot be reached in finite time by a nonconstant solution. After the crossing, $H<h$ prevents a later exit. The trajectory is bounded, and the convergence argument in Lemma~\ref{lem:convergence2saddle} applies. Its limiting energy is strictly less than $h$, excluding the boundary.
\end{proof}

The solutions move to the left when $v<0$ and to the right when $v>0$. The comparisons above apply on their common monotone arcs, as specified in the lemma; they do not assert a global ordering of arbitrary trajectories through all subsequent turns.

The next result will be based on the observation that if a solution $U(y)$
to equation \eqref{Agsystemb} crosses the $u$-axis at $\big( u(y_m),0 \big)^T$ that is not
an equilibrium point, then the convexity of the trajectory is given by
$$
        \frac{d^2 u}{d y^2} (y_m) = \partial_u G\big( u(y_m); u_0,a \big),
$$
and is therefore independent of $\alpha$. Using this fact, the monotonicity expressed in the inequalities of the monotonicity lemma
will be extended all the way to $v=0$, thereby demonstrating the strict monotonicity even at the intersection with
the $u$-axis. To help the reader remember the key geometrical feature of the result, it will
be referred to as the {\it convexity lemma}.

%

\begin{lemma}[convexity lemma]    \label{lem:convexity}
Assume that Maxwell's equal area rule holds, that the entropy is given by \eqref{entropy:explicit}, and
that the speed $a$ satisfies the prescribed speed assumptions~\eqref{hyp:speed-assumptions}. For a pair of diffusion-dispersion ratios
$$
     0 \leq \alpha_1 < \alpha_2,
$$
consider two solutions $v_1, v_2$ of the equation \eqref{eq:systemwrtu}
associated respectively to $\alpha_1, \alpha_2$. 
If there exists a non-equilibrium point $\overline{u}$ and $\epsilon>0$ for which
\begin{equation}    \label{eq:convexity.eq1}
          v_1(u) < v_2(u) < 0, \quad \forall u \in ] \overline{u}, \overline{u} +\epsilon[,
\end{equation}
(or $0<v_2(u) <v_1(u)$ for all $u \in ]\overline{u}-\epsilon,\overline{u} [$), then $v_1$ and $v_2$
cannot vanish simultaneously when $u=\overline{u}$.
\end{lemma}

\begin{proof}
We give the proof under the assumption \eqref{eq:convexity.eq1} and proceed by contradiction by assuming that
$v_1(\overline{u})=v_2(\overline{u})=0$. At the crossing with the $u$-axis, both solutions satisfy
\begin{equation}   \label{eq:convexity.eq2}
         \frac{d v_i}{dy}(\overline{y}_i) = \partial_u G(\overline{u};u_0,a),
\end{equation}
where $\overline{y}_i$ is such that $\overline{u}= u_i(\overline{y}_i)$. Since $\overline{u}$ is not an equilibrium, the right hand side of 
\eqref{eq:convexity.eq2} is
$$
    \kappa := \partial_u G(\overline{u}; u_0,a) \neq 0.
$$
Since as $y$ increases $u_i(y)$ decreases to $\overline{u}$ and $v_i(y)$ increases to $0$, we conclude geometrically 
that $dv_i/dy( \overline{y}_i)= \kappa >0$. Turning now to equation \eqref{eq:systemwrtu}, we find
$$
      \frac{d}{du} \bigg( \frac{1}{2} v_i^2(u) \bigg) \bigg|_{u = \overline{u} } = \kappa,
$$
so that for $u>\overline{u}$ but close to $\overline{u}$,
$$
      \frac{1}{2} v_i^2(u) \sim \kappa (u-\overline{u})  \quad \Longleftrightarrow \quad
        v_i(u) \sim - \sqrt{2 \kappa(u-\overline{u})}.
$$
Taking now the difference between the two equations  \eqref{eq:systemwrtu} for $v_i$, we find
\begin{align}  \nonumber
         \frac{1}{2} \frac{d}{du} \Big( v_1^2(u) - v_2^2(u) \Big) = & \, - \alpha_1 v_1(u) + \alpha_2 v_2(u) \\
             \simeq & \,  -(\alpha_1 - \alpha_2) \Big( - \sqrt{2 \kappa (u- \overline{u})} \Big)  < 0.    \label{eq:convexity.eq3}
\end{align}
We assumed \eqref{eq:convexity.eq1}, which meant
$$
        \frac{1}{2} \Big( v_1^2(u) - v_2^2(u) \Big) > 0
$$
for $u \in ]\overline{u}, \overline{u}+\epsilon[$. On the other hand, the two profiles vanish at $\overline{u}$, so integrating the negative derivative in \eqref{eq:convexity.eq3} from $\overline{u}$ to $u>\overline{u}$ makes this same difference negative for $u$ sufficiently close to $\overline{u}$. This contradiction excludes a common zero.

The positive case follows by the change of variables $(u,v)\mapsto(-u,-v)$, which preserves the diffusion parameter and converts $0<v_2<v_1$ into the negative ordering used above.
\end{proof}

\begin{corollary}   \label{cor:convexity}
Under the assumptions of the convexity lemma, suppose a nonconstant solution crosses the $u$-axis at $y=y^*$ with $u^*<u(y^*)<u_0$. If it is below the axis before this crossing, then
\be
            \label{eq:minimum-turn-intervals}
 u(y^*)\in(u^*,u_{K+1})\ \cup\!
 \bigcup_{k\in\{2,4,\ldots,K\}}(u_k,u_{k-1}).
\ee
If it is above the axis before the crossing, then
\be
\label{eq:maximum-turn-intervals}
 u(y^*)\in\bigcup_{k\in\{0,2,\ldots,K\}}(u_{k+1},u_k).
\ee
The union over $k\in\{2,4,\ldots,K\}$ is empty when $K=0$.
\end{corollary}

\begin{proof}
A nonconstant solution cannot reach an equilibrium in finite time. Hence at the crossing $v_y=G'(u(y^*))\neq0$. This derivative is positive for a crossing from below and negative for a crossing from above. The alternating signs of $G'$ on the listed intervals give the result, including the outer interval $(u^*,u_{K+1})$.
\end{proof}

%
%

We now present a few additional continuity results that will be critical later in the analysis.
In the next two lemmas, $U_\alpha$ is a family of nonconstant solutions whose initial value at a fixed finite time depends $C^2$ on $\alpha$. Saddle branches have this property after a regular choice of time origin.

\begin{lemma}    \label{lem:intersectionveq0}
Consider $U_\alpha(y) = \big(u_\alpha(y), v_\alpha(y) \big)^T $ a solution to the system \eqref{Agsystemb} 
parametrized by the diffusion-dispersion ratio $\alpha$ and suppose there is an
$\alpha_0$ for which $U_{\alpha_0}$
reaches the $u$-axis at a finite time $y=y_0$. Then for all
$\alpha$ in a neighborhood of $\alpha_0$, $U_\alpha$ is transverse to the $u$-axis at the crossing $u_\alpha\big( y_0(\alpha) \big)$
and
$$
       \alpha \longmapsto u_\alpha\big( y_0(\alpha) \big)
$$
is a $C^2$ parametrization of the intersection.
\end{lemma}

\begin{proof}
First of all, we remark that if $U_{\alpha_0}$ reaches the $u$-axis in a finite time $y_0$, then the
intersection $u_{\alpha_0}(y_0)$ is not an equilibrium point, i.e.
$$
       \partial_u G \big( u_{\alpha_0}(y_0); u_0, a  \big) \neq 0.
$$
The basic continuity result is that if $f$ is $C^k$ with respect to $\alpha$, then the solution $w(t,\alpha)$ to
$dw/dt(t,\alpha) = f\big( t,w(t,\alpha),\alpha \big)$ is $C^k$ with respect to $\alpha$. If we consider the function
\begin{align*}
           h: \mathbb{R}^2 & \longrightarrow \mathbb{R} \\
            (y,\alpha) & \longmapsto v_\alpha(y),
\end{align*}
then its zero set describes the set of all $(y,\alpha)$ for which $U_\alpha$ crosses the $u$-axis.
We know that $(y_0,\alpha_0)$ is a root of $h$ yet
$$
     \frac{\partial h}{\partial y}(y_0,\alpha_0) = \frac{d v_{\alpha_0}}{dy}(y_0) 
             = \partial_u G \big( u_{\alpha_0}(y_0); u_0,a \big) \neq 0.
$$
The implicit function theorem shows that the zero set containing $(y_0,\alpha_0)$
can be parametrized by a $C^2$ function $\alpha \mapsto y_0(\alpha)$, at least for $\alpha$
in a small neighborhood of $\alpha_0$. In particular, $\alpha \mapsto u_\alpha( y_0(\alpha))$
is also $C^2$.
\end{proof}

The next result is similar to the previous one but concerns intersections with a vertical line $u = \overline{u}$.

\begin{lemma}    \label{lem:intersectionueq0}
Consider $U_\alpha(y) = \big(u_\alpha(y), v_\alpha(y) \big)^T$ a solution to
the system \eqref{Agsystemb} and suppose that for some $\bar{u}$ there is a diffusion-dispersion 
ratio $\overline{\alpha}$ for which $U_{\overline{\alpha}}$ reaches  $u_{\overline{\alpha}}(\bar{y})=\bar{u}$ for
some $\bar{y}$. If $v_{\overline{\alpha}}(\bar{y}) \neq 0$, then for all $\alpha$ in a neighborhood
of $\overline{\alpha}$, $U_\alpha$ is transverse to the line $u = \bar{u}$ and 
\begin{align*}
       \alpha \longmapsto & \, \overline{y}(\alpha), \\
       \alpha \longmapsto & \, v_\alpha\big( \bar{y}(\alpha) \big),
\end{align*}
are $C^2$. 
\end{lemma}

\begin{proof}
The proof is very similar to the previous one. In this case, the set of all $(y,\alpha)$
for which $u_\alpha(y) = \bar{u}$ is the zero set of
\begin{align*}
           h: \mathbb{R}^2 & \longrightarrow \mathbb{R} \\
            (y,\alpha) & \longmapsto u_\alpha(y) - \bar{u}.
\end{align*} 
We note that because of our assumption that $v_{\overline{\alpha}}(\bar{y}) \neq 0$, then
$$
     \frac{\partial h}{\partial y}(\bar{y}, \overline{\alpha}) = \frac{d u_{\overline{\alpha}} }{dy}( \bar{y}) 
           = v_{\overline{\alpha}}(\bar{y}) \neq 0.
$$
The implicit function theorem therefore states that the zero set near $(\bar{y},\overline{\alpha})$
can be parametrized by a  $C^2$ function $\alpha \mapsto \bar{y}(\alpha)$. The function
$\alpha \mapsto v_\alpha\big(  \bar{y}(\alpha) \big)$ is also $C^2$.
\end{proof}

Although the solutions rotate clockwise, we have never needed to account for the number
of crossings made by one of these solutions. We correct the situation with the following definition.

\begin{definition}
A traveling wave solution $U(y) = \big( u(y), v(y) \big)^T $ of \eqref{Agsystemb} is said to have
$n$ {\bf oscillations} if there exists exactly $n$ points $y_1> \cdots > y_n$
such that $u$ is strictly monotone between these points and alternates
between increasing and decreasing over the intervals delimited by these points. 
\end{definition}

We now introduce two additional trajectories that will help us to understand what happens to trajectories 
for $\alpha$ strictly above and below a critical value $\alpha_n(k)$ at which a trajectory
leaving $u_0$ reaches a saddle equilibria $u_k$ after $n$ oscillations.

\begin{definition}   \label{defn:extendedtrajectory}
Suppose that one has a value ${\alpha}$ of the diffusion-dispersion ratio for which there is
an $n$-oscillation trajectory $U_{{\alpha}}$  leaving $u_0$ in $Q^3_0$ and reaching the saddle equilibria $u_k$ through $Q^\ell_k$.
If $n$ is even, then $\ell$ must be $4$ and one defines 
$$
     \underline{\ell} = 1\,, \quad \overline{\ell} = 3,
$$
while if $n$ is odd, then $\ell$ must be $2$ and one defines
$$
     \underline{\ell} = 3\,, \quad \overline{\ell} = 1.
$$
One also defines the inner extension to $C^3_{0,\alpha}$ to be
$$
       \underline{C}^3_{0,{\alpha}} := C^{\underline{\ell}}_{k,{\alpha}},
$$
and the outer extension to be
$$
       \overline{C}^3_{0,{\alpha}} :=  C^{\overline{\ell}}_{k,{\alpha}}.
$$
\end{definition}


\begin{lemma}[extension lemma]   \label{lem:extendedtrajectory}
Suppose that at $\alpha_0>0$ the trajectory $C^3_{0,\alpha_0}$ connects the hyperbolic saddles $u_0$ and $u_k$. Let $U_\alpha$, $\underline U$, and $\overline U$ denote solutions corresponding to $C^3_{0,\alpha}$ and its inner and outer extensions at $\alpha_0$. For every fixed $y\in\RR$ and $\epsilon>0$ there is $0<\delta<\alpha_0$ such that both conclusions hold:
\begin{itemize}
\item for every $\alpha\in(\alpha_0,\alpha_0+\delta)$, some time $y_{\epsilon,\alpha}$ satisfies $\|\underline U(y)-U_\alpha(y_{\epsilon,\alpha})\|<\epsilon$;
\item for every $\alpha\in(\alpha_0-\delta,\alpha_0)$, some time $\bar y_{\epsilon,\alpha}$ satisfies $\|\overline U(y)-U_\alpha(\bar y_{\epsilon,\alpha})\|<\epsilon$.
\end{itemize}
\end{lemma}

\begin{proof}
Set $q(u)=\partial_uG(u)$ and $X_\alpha(u,v)=(v,q(u)-\alpha v)^T$. Choose a vertical section $u=\sigma$ on the final monotone segment of the connecting orbit, and put its crossing at $y=0$. Continue the incoming branch from $u_0$ and the stable branch of $u_k$ to this section for $\alpha$ near $\alpha_0$. Their heights $v^u_\alpha(\sigma)$ and $v^s_\alpha(\sigma)$ depend differentiably on $\alpha$. Define $\Delta(\alpha)=v^u_\alpha(\sigma)-v^s_\alpha(\sigma)$ and write $v_0=v_{\alpha_0}(0)\neq0$.

We first compute the sign of the splitting. A variation $w=\partial_\alpha U_\alpha$ along either branch satisfies $w_y=DX_{\alpha_0}w+(0,-v)^T$. For $W=\det(X_{\alpha_0},w)$, differentiation gives
\be
\label{eq:determinant-variation}
 W_y=-\alpha_0W-v^2,\qquad (e^{\alpha_0y}W)_y=-e^{\alpha_0y}v^2.
\ee
On the unstable branch $e^{\alpha_0y}W\to0$ as $y\to-\infty$; on the stable branch it tends to zero as $y\to+\infty$. These limits follow from the saddle eigenvalues: $\lambda_+>0$ at $u_0$, and $\lambda_-<-\alpha_0$ at $u_k$. Parameter differentiation can introduce a factor $1+|y|$, which preserves exponential decay. At the section the variations have zero $u$-component, so subtraction yields
\be
\label{eq:shooting-transversality}
 v_0\Delta'(\alpha_0)=-\int_{-\infty}^{\infty}e^{\alpha_0y}v(y)^2\,dy<0.
\ee
The integral is finite by the same eigenvalue estimates. In particular the zero of the continued shooting function is simple.

If the number of turns is even, $v_0<0$, hence $\Delta'(\alpha_0)>0$. Increasing $\alpha$ places the incoming orbit above the stable curve and sends it to the right outgoing branch $C^1_{k,\alpha}$; decreasing $\alpha$ sends it to $C^3_{k,\alpha}$. If the number is odd, $v_0>0$ and the signs reverse: the respective outgoing branches are $C^3_{k,\alpha}$ and $C^1_{k,\alpha}$. These are precisely the inner and outer extensions in the definition above.

We justify convergence along those outgoing branches. Take the local stable and unstable curves used to define the saddle separatrices as coordinate axes $(s,t)$, with $s$ the stable coordinate. In a sufficiently small square $|s|,|t|\leq\rho$, invariance of the axes gives
\be
\label{eq:local-saddle-coordinates}
 s_y=(-\beta_\alpha+r_s(s,t,\alpha))s,\qquad
 t_y=(\gamma_\alpha+r_t(s,t,\alpha))t,
\ee
where $\beta_\alpha,\gamma_\alpha>0$ and $r_s,r_t$ tend to zero at the origin, uniformly near $\alpha_0$. Thus there are constants $m,M>0$, independent of nearby $\alpha$, such that $|s|$ decreases at rate at least $m$ and $|t|$ increases at a rate between $m$ and $M$. On an incoming section $|s|=\rho$, the transverse coordinate is $t=\eta_\alpha\to0$, with the sign just determined by \eqref{eq:shooting-transversality}. Until the exit $|t|=\rho$, the residence time $T_\alpha$ and stable coordinate satisfy
\be
\label{eq:saddle-exit-estimate}
 T_\alpha\geq\frac1M\log\frac{\rho}{|\eta_\alpha|},\qquad
 |s_{\rm exit}|\leq\rho\left(\frac{|\eta_\alpha|}{\rho}\right)^{m/M}\longrightarrow0.
\ee
The exit points therefore converge to the appropriate unstable separatrix at $\alpha_0$. For any prescribed point on that separatrix we choose the square sufficiently small that the point lies after the exit. Continuous dependence over the remaining finite time interval, followed by a time translation, proves both assertions.
\end{proof}


\section{General construction}
\label{sec:general}

In this section, we study saddle connections and attraction intervals for the traveling-wave equation as the diffusion-dispersion ratio varies.  
The main conclusion of the construction is that
it may not be possible to find travelling waves connecting $u_0$ to every intermediate
saddle equilibrium, and the study will identify conditions under
which this may occur.

We begin by treating a generic flux whose bubbles form an elementary bowtie. This will become the simplest case of an 
analysis involving nested bowties, but it is a good way to build our understanding.
In this elementary case $K\geq2$. We will be studying an elementary isolated bowtie, that is we will
initially be assuming that the generic flux and the quadratic entropy produce a sequence of $K+2$
equilibrium points whose initial configuration of bubbles are $RD \cdots DL$. So for all $k \in \{ 3,5, \ldots, K-1 \}$,
the intervals $]u_{k+1}, u_{k-1}[$ will be dual-constrained and
$$
     C^1_{k+1,0} = C^2_{k-1,0} \quad \text{and} \quad C^3_{k-1,0} = C^4_{k+1,0},
$$
while $C^3_{K,0}$ and $C^1_{2,0}$ will each contain a single equilibrium point.

The next three existence statements and Theorem~\ref{thm:bowtie1} are proved together after the theorem, using continuation up to the first failure to pass the next saddle. This avoids presuming that every shooting branch crosses a chosen section for all $\alpha$.

\begin{lemma}  \label{lem:bowtie1}
Suppose that $C^3_{0,0}$ contains only an elementary bowtie. Then there exists a 
diffusion-dispersion ratio $\alpha_0(2)$ for which there is a $0$-oscillation trajectory connecting $u_0$ to $u_2$.
\end{lemma} 


\begin{lemma} \label{lem:bowtie2}
Suppose that $C^3_{0,0}$ contains only an elementary bowtie  and that an $n$-oscillation trajectory
connects $u_0$ to a saddle equilibrium $u_k$ for some diffusion-dispersion ratio $\alpha_n(k)$. 
If $n$ is even and $k \in \{ 2, \ldots, K-2\}$, then there exists an $n$-oscillation trajectory connecting $u_0$ to $u_{k+2}$
for a diffusion-dispersion ratio $\alpha_n(k+2) < \alpha_n(k)$.  
If $n$ is odd and $k \in \{ 4, \ldots, K\}$, then there exists an $n$-oscillation trajectory connecting $u_0$ to $u_{k-2}$
for a diffusion-dispersion ratio $\alpha_n(k-2) < \alpha_n(k)$.
\end{lemma}

\begin{lemma}  \label{lem:bowtie3}
Suppose that $C^3_{0,0}$ contains only an elementary bowtie. 
Suppose $n$ is even and there exists an $n$-oscillation trajectory connecting $u_0$ to $u_K$,
then there exists $\alpha_{n+1}(K)< \alpha_n(K)$ for which there is an $(n+1)$-oscillation
trajectory connecting $u_0$ to $u_K$. 
Suppose $n$ is odd and there exists an $n$-oscillation trajectory connecting $u_0$ to $u_2$, then
there exists $\alpha_{n+1}(2) < \alpha_n(2)$ for which there exists an $(n+1)$-oscillation 
trajectory connecting $u_0$ to $u_2$.
\end{lemma}

We arrive at one of our main conclusions. 

\begin{theorem}   \label{thm:bowtie1}
Consider a generic flux $f$, a quadratic entropy $U$,  the speed $a$ satisfying the prescribed speed assumptions~\eqref{hyp:speed-assumptions}, and that $C^3_{0,0}$ contains only an elementary bowtie.
All trajectories in this statement are the left outgoing branch $C^3_{0,\alpha}$. 
For every number $n$ of oscillations and every even $k \in \{2, 4, \ldots, K \}$, there exists
a unique diffusion-dispersion ratio $\alpha_n(k)$ for which there is an $n$-oscillation 
trajectory connecting $u_0$ to $u_k$. Furthermore, one has: 
\begin{itemize}
\item[1.] if $\alpha > \alpha_0(2)$, then this branch converges to $u_1$;
              
\item[2.] if $n$ is even and $\alpha \in] \alpha_{n+1}(K),  \alpha_n(K) [$ then this branch converge to $u_{K+1}$, while if $n$ is odd 
           and $\alpha \in ] \alpha_{n+1}(2),  \alpha_n(2)[$
              then this branch approach $u_1$;              
              
\item[3.] if $n$ is even and $\alpha \in ]\alpha_n(k+2), \alpha_n(k)[$, then this branch converges to $u_{k+1}$;
              if $n$ is odd and $\alpha \in ]\alpha_n(k), \alpha_n(k+2)[$, then this branch converges to $u_{k+1}$;  
              
\item[4.] the diffusion-dispersion ratios $\alpha_n(k)$ are unique;

\item[5.] for any fixed $k$, the sequence of diffusion-dispersion ratios $\alpha_n(k)$ tend to zero as $n \to \infty$.
\end{itemize}
\end{theorem}

\begin{proof}
We prove the theorem and the preceding three lemmas by a single continuation argument. 
Since the internal saddles belong to one elementary bowtie, their heights are equal:
\be
\label{eq:common-saddle-height}
 h=G(u_2)=G(u_4)=\cdots=G(u_K)>0.
\ee
On $[u_K,u_2]$, $G\geq h$, with strict inequality away from the saddles. On the left and right outer intervals there is one stable equilibrium each. While $-h<H<0$, a trajectory can turn only to the left of $u_{K+1}$ or to the right of $u_1$: a turn in an internal bubble would have $H=-G(u)\leq-h$. Once $H<-h$, the trajectory lies inside an elementary bubble and converges to its unique stable equilibrium by absorption.

First fix a finite number $N$ of monotone segments. Put $L=u_0-u^*$ and $M=\sqrt{2\max_{[u^*,u_0]}G}$. On every such segment, including a segment ending asymptotically at an equilibrium,
\be
\label{eq:monotone-segment-loss}
 \int v^2\,dy=\int |v|\,|du|\leq ML.
\ee
Consequently, for $\alpha NML<h$ the first $N$ segments cannot end at an internal equilibrium: their total energy loss is less than $h$, whereas every such equilibrium has $G\geq h$. Confinement and convergence exclude any other termination. Thus, for each fixed finite itinerary along the outer region, sufficiently small positive $\alpha$ produces all its turns and passes all its saddle sections with nonzero velocity.

We enumerate the successive target pairs by
\be
\label{eq:successive-target-order}
 (0,2),(0,4),\ldots,(0,K),(1,K),(1,K-2),\ldots,(1,2),(2,2),\ldots.
\ee
Passing a target $(n,k)$ means crossing $u=u_k$ with nonzero velocity on the segment after exactly $n$ turns, before capture in an elementary bubble. At such a crossing $H=v^2/2-h>-h$. The small-parameter argument gives an interval of passing parameters adjacent to zero for every fixed target.

For the first target, passing fails for sufficiently large $\alpha$. Indeed, the incoming stable branch of $u_2$, denoted $v^s_\alpha(u)<0$, extends backwards to $u_1$: on $(u_2,u_1)$ the equation $v\,v_u=-\alpha v+G'(u)$ has $G'(u)>0$, so this branch cannot vanish there. It cannot become unbounded at finite $u$, since $v_u=-\alpha+G'(u)/v$ is bounded once $v$ stays away from zero. Integrating its equation in time gives
\be
\label{eq:alpha0bound}
 v^s_\alpha(u_1)<-\alpha(u_1-u_2),\qquad |v^u_\alpha(u_1)|\leq M
\ee
whenever the outgoing branch reaches $u_1$. If it approaches $u_1$ asymptotically, passing already fails. For $\alpha>M/(u_1-u_2)$ the outgoing branch lies above the incoming stable branch at $u_1$ and cannot cross it; it therefore cannot pass $u_2$.

We now give the induction step, which also applies to this first target with no previous finite upper endpoint. Suppose all earlier targets have been constructed, the last at $\beta_{r-1}$, and every $0<\alpha<\beta_{r-1}$ passes the preceding target. Between that target and the next there is exactly one elementary bubble: a dual-constrained bubble, or an outer bubble if the next target requires an additional turn. The extension lemma shows that for $\alpha<\beta_{r-1}$ sufficiently close to $\beta_{r-1}$ the trajectory enters this bubble and is captured by its unique sink. Hence it fails to pass the next target.

Let $P_r\subset(0,\beta_{r-1})$ be the parameters passing that target. It is open by transversality of a finite crossing, contains an interval adjacent to zero, and omits an interval below $\beta_{r-1}$. Every boundary point of $P_r$ in this parameter interval is a connection to the target saddle. To see this, the trajectory there has already passed the previous target. Before passing the next one, its only other possible terminal equilibrium is the unique intervening sink. Capture by that sink occurs in a finite-time attracting neighborhood and is open in $\alpha$, so it cannot occur at a boundary of $P_r$. Passing the target with nonzero velocity is also open. Convergence therefore leaves precisely a connection to the target. In an outer bubble the same argument includes the one required turn; a turn in an internal bubble would already give $H<-h$ and irreversible capture.

At each such boundary connection, the signed formula~\eqref{eq:shooting-transversality} places passing parameters immediately below the critical ratio and nonpassing parameters immediately above it. Thus every boundary point is a right endpoint of a component of $P_r$. An additional component would have a positive left endpoint, with the opposite orientation, which is impossible. Therefore
\be
\label{eq:passing-parameter-interval}
 P_r=(0,\beta_r),\qquad 0<\beta_r<\beta_{r-1},
\ee
and $\beta_r$ is the unique connection to this target in the remaining parameter range. For the initial target the large-parameter bound replaces the omitted interval below $\beta_{r-1}$ and gives the same conclusion. Every parameter between consecutive critical ratios is captured by the unique intervening sink. This proves the three preceding lemmas and assertions 1--4 of the theorem, with $\beta_r=\alpha_n(k)$ in the displayed order. It also excludes additional connections in the ranges already classified.

Finally, choose a closed interval $J\subset(u_2,u_1)$ and put
\be
\label{eq:turn-loss-bound}
 A=\int_J\sqrt{2(G(u)-h)}\,du>0.
\ee
On a connection with $n$ turns, every finite point has $H>-h$. Between consecutive outer turns the trajectory crosses $J$ and satisfies $|v|>\sqrt{2(G-h)}$. The total energy identity therefore gives, for $n\geq2$,
\be
\label{eq:critical-turn-loss}
 h=\alpha_n(k)\int_{-\infty}^\infty v^2\,dy
 \geq\alpha_n(k)(n-1)A.
\ee
It follows that $\alpha_n(k)\leq h/((n-1)A)\to0$, uniformly in the finitely many targets $k$. The consecutive parameter intervals consequently exhaust $(0,\infty)$, completing the classification and the proof.
\end{proof}

The previous theorem is essentially a generalization of Theorem 4.1 from \cite{BedChaCoqLeF}, except in their 
case, the elementary bowtie was simply an $RL$ while we have extended the result to include an arbitrary number of 
intermediate saddle equilibrium point.

The previous theorem might give the reader the impression that it is always possible to construct an
$n$-oscillation trajectory from $u_0$ to $u_{k+2}$ when it is known that an
$n$-oscillation trajectory from $u_0$ to $u_k$ exists. Unfortunately, this is not the
case as the following two examples illustrate. In the following examples, we will be using 
the notation $\overline{\alpha}_n(j,k)$ to denote a 
diffusion-dispersion ratio for which there is an $n$-oscillation trajectory connecting 
$u_j$ to $u_k$.

The next two examples are numerical illustrations. Their nonexistence conclusions use the strict threshold inequalities indicated by the computations; rigorous numerical enclosures of these thresholds have not been supplied.

\begin{figure}[!htb] 
\centering
\begin{minipage}{.5\textwidth}
       \centering  \includegraphics[width=\linewidth]{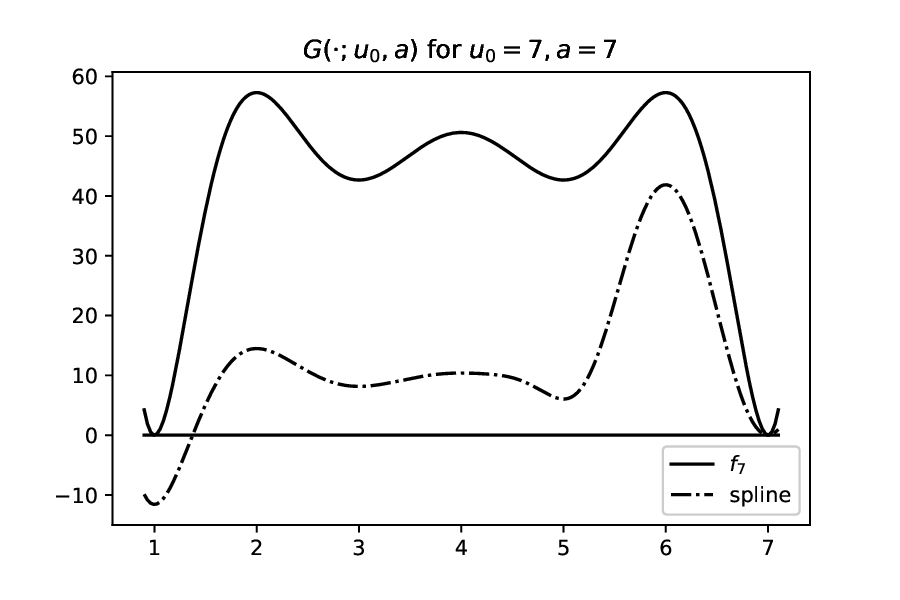}
\end{minipage}%
\begin{minipage}{.5\textwidth}
       \centering  \includegraphics[width=\linewidth]{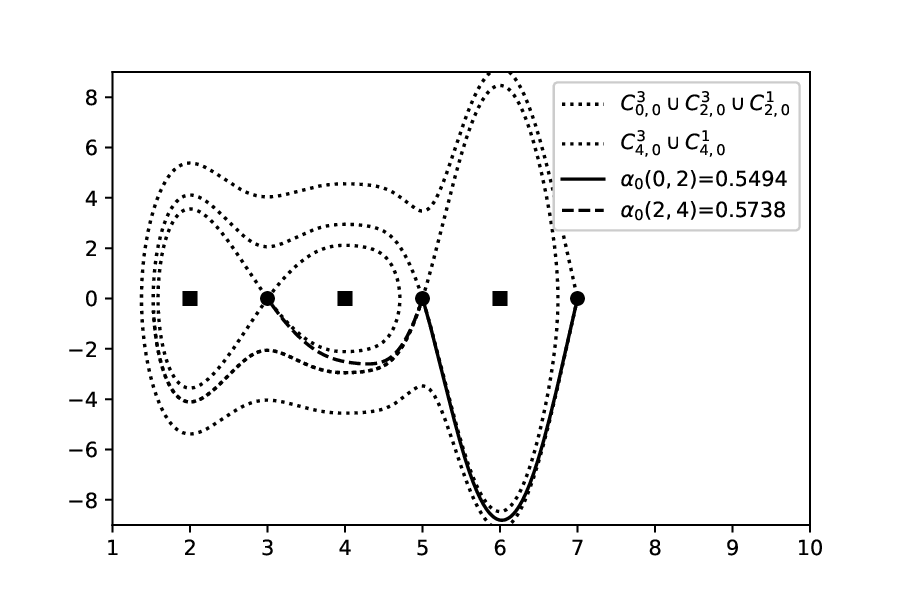}
\end{minipage} 
\caption{For the spline described in Example \ref{ex:fluxes2}, the function $G(\cdot \, ; 7,7)$ is seen on the left
and the graph of some of its trajectories are presented on the right. All the zero diffusion trajectories
originating from one of the saddle equilibrium are seen as dotted lines. There exists
a trajectory connecting $u_0$ to $u_2$ for some value $\alpha_0(2)=\overline{\alpha}_0(0,2)$
while the trajectory connecting $u_2$ to $u_4$ exists for a larger value of $\overline{\alpha}_0(2,4)$. 
The numerical ordering indicates that there is no $0$-oscillation trajectory connecting $u_0$ to $u_4$.}
\label{fig:counter1}
\end{figure}

\begin{example}   \label{ex:counter1}
The following simple example shows that even if we can find an $n$-oscillation trajectory connecting 
$u_0$ to some saddle equilibria $u_k$, it does not imply that we find an $n$-oscillation
trajectory connecting $u_0$ to $u_{k+2}$.  We consider
the flux described by the cubic spline of Example \ref{ex:fluxes2}. 
In Figure \ref{fig:counter1}, we present both the graph of its function $G$, for the entropy $U(u) = u^2/2$ and $u_0=7, a=7$,
and some of the trajectories to the system \eqref{Agsystemb}. Numerical computations
show that $\alpha_0(2) \approx 0.54935$ but $\overline{\alpha}_0(2,4) \approx 0.57380$. Assuming the strict ordering indicated by these values, the branch $C^3_{2,\alpha_0(2)}$ passes $u_4$ with $v<0$, since $\alpha_0(2)<\overline\alpha_0(2,4)$. For $0<\alpha<\alpha_0(2)$, the first negative arc of $C^3_{0,\alpha}$ lies below this branch and also passes $u_4$. For $\alpha>\alpha_0(2)$, it is captured in the right bubble before reaching $u_2$. At equality it ends at $u_2$. Thus the numerical thresholds indicate the absence of a $0$-oscillation connection from $u_0$ to $u_4$.

The geometry of this example contains a bowtie $RL$ connected to $u_4$ within 
a larger bowtie $RL$ connected to $u_2$. In this sense, the result is entirely
consistent with Theorem \ref{thm:generalbowtie}. 
\end{example}

\begin{figure}[!htb] 
       \centering  \includegraphics[width=\linewidth]{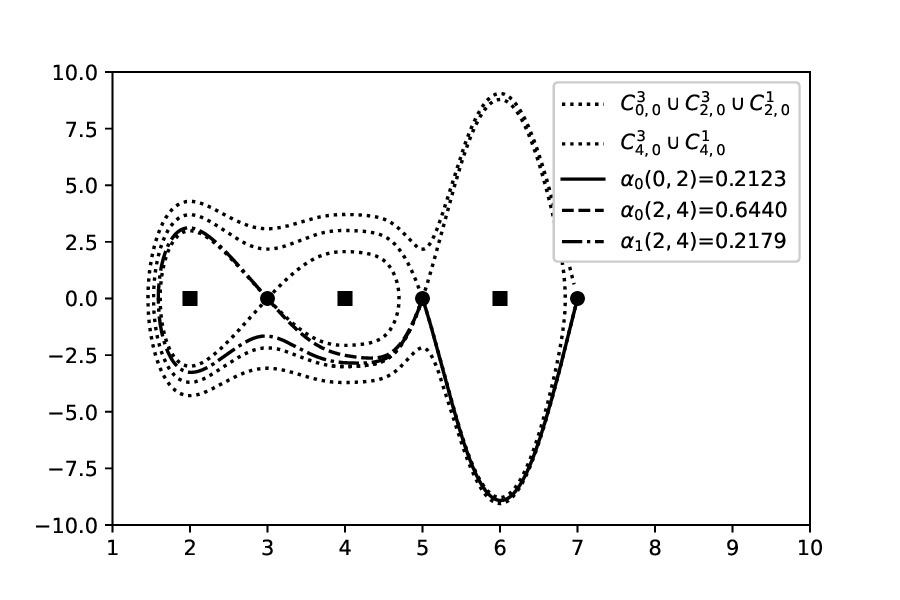}
\caption{For the spline described in Example \ref{ex:counter2}, we present some of its
trajectories in phase space. All the zero diffusion trajectories
originating from one of the saddle equilibrium are seen as dotted lines.  The figure
shows that $\alpha_0(2)$ exists, as do the $0$ and $1$ oscillation trajectories connecting $u_2$ to $u_4$,
but because $\alpha_0(2)$ is smaller than both $\overline{\alpha}_0(2,4)$ and $\overline{\alpha}_1(2,4)$,
the numerical ordering indicates no $1$-oscillation trajectory connecting $u_0$ to $u_4$.} 
\label{fig:counter2}
\end{figure}

\begin{table}[ht!]
\centering
\caption{Values determining the cubic spline of Example \ref{ex:counter2}.}
\label{tab:spline2}
\begin{tabular}{|c||c|c|c|c|c|c|}
\hline
$u$       & 1.5  &  2.5  &  3.5  &  4.5  &  5.5  &  6.5 \\ \hline \hline
$f(u)$    & 45   &  13   &  30   &   29  & 105  & -18 \\ \hline
\end{tabular}
\end{table}

\begin{example}   \label{ex:counter2}
This example is a slight extension of the previous result that shows that it may not be possible to 
reach $u_4$ even after allowing an oscillation, or in other words, oscillations don't
make it any easier to connect to the next saddle equilibrium.
We consider a perturbation of the flux described by a cubic spline in Example \ref{ex:fluxes2}
but where we now use the interpolated values from Table \ref{tab:spline2}. The values
of the first derivative at $u=1$ and $u=7$ are unchanged from Example \ref{ex:fluxes2}. 
The graph  of $G$, for the entropy $U(u) = u^2/2$ and $u_0=7, a=7$, is almost identical to the graph 
of $G$ seen in Figure \ref{fig:counter1}, and is therefore omitted. In Figure \ref{fig:counter2}, we present 
some of the trajectories in phase space to the system \eqref{Agsystemb}. Numerical computations
show that $\alpha_0(2)  \approx 0.21231$ but $\overline{\alpha}_0(2,4) \approx 0.64404$
and $\overline{\alpha}_1(2,4) \approx 0.21787$. For the strict ordering indicated by these values, 
$C^3_{2,\alpha_0(2)}$ first passes $u_4$ with $v<0$ and, after its left turn, returns past $u_4$ with $v>0$, 
because $\alpha_0(2)$ is below both internal thresholds. For $0<\alpha<\alpha_0(2)$, comparison on the 
negative arc puts the first turn of $C^3_{0,\alpha}$ farther to the left; the convexity lemma excludes a 
common turn. Comparison of the returning positive arcs then puts it above the returning arc from $u_2$. 
It also passes $u_4$ with $v>0$. For $\alpha>\alpha_0(2)$ it remains in the right bubble, and at equality 
it ends at $u_2$. The computations therefore indicate the absence of a $1$-oscillation connection from $u_0$ to $u_4$.
\end{example}

The next step is to give a general description of all solutions to  \eqref{Agsystemb} that converge to
a saddle equilibrium as a function of the diffusion-dispersion ratio. Our earlier analysis of the nestings of bubbles 
has shown that all configurations
$$
     RP_{K-1} \cdots P_3L \longrightarrow \cdots \longrightarrow R D \cdots D L
$$
can be reduced to a single outer bowtie that contains possible nestings of bowties inside 
of each of its bubbles. Each step in the reduction is the identification of a single bubble in which an isolated 
bowtie can be found.

In a nested configuration, a trajectory entering a bubble remains there. A connection to a saddle on its boundary gives, through the extension lemma, a one-sided open interval of trajectories entering that bubble. The following three lemmas establish this local entry property. The parameter classification in Theorem~\ref{thm:generalbowtie} then applies on every open interval where entry has been established.

\begin{lemma}    \label{lem:generalRbubble}
Suppose $R$ is the bubble bounded by $C^3_{k,0}=C^2_{k,0}$, and $C^3_{0,\alpha_0}$ connects $u_0$ to $u_k$ after an even number of turns, with $\alpha_0>0$. There exists $\delta>0$ such that, for $\alpha\in(\alpha_0-\delta,\alpha_0)$, the left outgoing trajectory from $u_0$ enters $R$ and converges to an equilibrium in its interior.
\end{lemma}
\begin{proof}
The outer extension $C^3_{k,\alpha_0}$ enters $R$: its energy decreases strictly from the boundary level $-G(u_k)$. Choose a point of this extension strictly inside $R$. Lemma~\ref{lem:extendedtrajectory} places a point of $C^3_{0,\alpha}$ inside $R$ for every $\alpha<\alpha_0$ sufficiently close to $\alpha_0$. The absorbing-region corollary then gives convergence.
\end{proof}

\begin{lemma}     \label{lem:generalLbubble}
Suppose $L$ is the bubble bounded by $C^1_{k,0}=C^4_{k,0}$, and $C^3_{0,\alpha_0}$ connects $u_0$ to $u_k$ after an odd number of turns, with $\alpha_0>0$. There exists $\delta>0$ such that, for $\alpha\in(\alpha_0-\delta,\alpha_0)$, the left outgoing trajectory from $u_0$ enters $L$ and converges to an equilibrium in its interior.
\end{lemma}
\begin{proof}
The outer extension is now $C^1_{k,\alpha_0}$. It enters the interior of $L$ by strict energy dissipation, and the same application of Lemma~\ref{lem:extendedtrajectory} proves the assertion.
\end{proof}

\begin{lemma}     \label{lem:generalDbubble}
Suppose $D$ is the bubble bounded by the two zero diffusion connections between saddles $u_k<u_j$. If $C^3_{0,\alpha_0}$ reaches $u_j$ after an even number of turns, or $u_k$ after an odd number, with $\alpha_0>0$, then for every $\alpha<\alpha_0$ sufficiently close to $\alpha_0$ the left outgoing trajectory from $u_0$ enters $D$ and converges to an equilibrium in its interior.
\end{lemma}
\begin{proof}
In the first case the outer extension is $C^3_{j,\alpha_0}$, and in the second it is $C^1_{k,\alpha_0}$. Each enters $D$ by strict energy dissipation. Lemma~\ref{lem:extendedtrajectory} and absorption give the assertion.
\end{proof}

We take a moment to introduce some notation. Since the trajectories
essentially spiral clockwise, there is a natural order to the
potential trajectories. We say that $(n,k)$, formally representing the potential $n$-oscillation
trajectory connecting $u_0$ to $u_k$,  is {\bf before} $(n',k')$ if either
\begin{itemize}
\item $n<n'$, or
\item $n=n'$ is even and $k'> k$, or
\item $n=n'$ is odd and $k> k'$.
\end{itemize}
In practice, this gives the following order to these pairs of indices :
$$
   (0,2), (0,4), \ldots, (0,K), (1,K), (1,K-2), \ldots, (1,2), (2,2), (2,4), \ldots
$$
We use this convention in the statement of the following result.

For the general classification, we write $q=G'$ and recall $H=v^2/2-G(u)$. A bubble $B$ is a bounded connected component of $\{H<-g_B\}$, where $g_B$ is its boundary height. If $B$ contains saddles, we denote by $h_B$ the least value of $G$ at its interior saddles, and by $p_1>\cdots>p_m$ the saddles attaining this value. The $m+1$ connected components of $\{H<-h_B\}$ inside $B$ are its child bubbles, ordered from left to right.

\begin{theorem}   \label{thm:generalbowtie}
Fix the speed $a$ satisfying the prescribed speed assumptions~\eqref{hyp:speed-assumptions} and an open interval $I\subset(0,\infty)$ such that $C^3_{0,\alpha}$ enters $B$ for every $\alpha\in I$. If $B$ contains no saddle, every such trajectory converges to its unique stable equilibrium. Otherwise the following assertions hold.
\begin{enumerate}
\item The connections from $u_0$ to the saddles $p_i$ form a consecutive block in the order introduced above. More precisely, the index of a connection with $n$ turns is
\be
\label{eq:general-itinerary-index}
 j(n,i)=
 \begin{cases}
 mn+i-1,&n\ \hbox{even},\\
 m(n+1)-i,&n\ \hbox{odd}.
 \end{cases}
\ee
Each index is attained at most once in $I$, at a parameter $\beta_j$. The set of attained indices is an interval of nonnegative integers, possibly empty or infinite, and $j<k$ implies $\beta_k<\beta_j$. These critical parameters are locally finite in $I$.
\item Every component of $I$ minus these critical parameters is an open interval on which the trajectories enter a fixed child bubble. Between two consecutive attained indices it is $(\beta_{j+1},\beta_j)$; end intervals are truncated by $I$. Consecutive arrivals at distinct saddles delimit the intervening $D$ bubble. The transition from negative to positive arrival at $p_m$ gives the left $R$ bubble, and that from positive arrival at $p_1$ to the next negative arrival gives the right $L$ bubble.
\item Repeating this classification in the child bubbles gives an exhaustive partition for all $\alpha\in I$ after finitely many descents for each parameter. Each trajectory either reaches a saddle at a critical parameter or converges to the unique stable equilibrium of a bubble without saddles. The full set $\Sigma_B$ of parameters for convergence to any interior saddle of $B$ is relatively closed and at most countable. Its complement consists of open intervals with constant limiting stable equilibrium. Every saddle connection has finitely many turns and is a simple shooting zero. Its critical parameter is unique on each fixed succession of entry intervals with a fixed terminal index.
\end{enumerate}
\end{theorem}

\begin{proof}
We first justify the hierarchy of bubbles. The projection of $B$ is an interval on which $G>g_B$. When there is an interior saddle, $h_B>g_B$, and every interior maximum of $G$ is greater than $h_B$: each such maximum is adjacent to an interior minimum on at least one side. Between successive $p_i$, we have $G>h_B$, since an additional contact with this level would be another minimum attaining $h_B$. Lemma~\ref{lem:bubble} therefore connects successive $p_i$ by zero diffusion separatrices. The outer branches turn at the first regular contacts with $G=h_B$, before the boundary of $B$. Their loops and the intervening connections bound exactly the $m+1$ child bubbles. Every saddle interior to a child has height greater than $h_B$. At least one interior saddle becomes a boundary point at each descent, so this hierarchy is finite. Without an interior saddle, alternation leaves exactly one stable equilibrium.

We next construct a common section. We write $X_\alpha\coloneqq(v,q-\alpha v)$, fix $g_B<e<h_B$, and define the energy oval and its Hamiltonian parametrization by
\be
\label{eq:general-energy-section}
 \Gamma_e\coloneqq\{(u,v)\in B:H(u,v)=-e\},\qquad
 \gamma_e'(\theta)=X_0(\gamma_e(\theta)),\quad
 \gamma_e(\theta+T_e)=\gamma_e(\theta),
\ee
where $T_e>0$ is the period. All equilibrium heights in $B$ are at least $h_B$, so $\Gamma_e$ is a single simple regular oval. We place $\theta=0$ at its right turning point; its negative and positive branches have phases $(0,T_e/2)$ and $(T_e/2,T_e)$, respectively.

For each $\alpha\in I$, the outgoing branch crosses $\Gamma_e$ exactly once. Indeed, absorption and Lemma~\ref{lem:convergence2saddle} give an interior equilibrium limit of energy at most $-h_B$, and $H$ decreases strictly along every nonconstant orbit. Both stable branches of every $p_i$ also cross $\Gamma_e$ exactly once in backward time, for every $\alpha>0$. Until this crossing, they stay in a compact subset of $B$ with $-h_B<H\leq-e$. If a branch never crossed, the backward energy identity, bounded derivatives, and the convergence argument of Lemma~\ref{lem:convergence2saddle} would give an equilibrium limit of energy strictly greater than $-h_B$, which is impossible. These finite crossings depend continuously on $\alpha$, also at $v=0$, by strict monotonicity of $H$ and continuous dependence on finite time intervals.

The stable crossing points never coincide, by uniqueness at fixed $\alpha$. Their cyclic order is
\be
\label{eq:general-stable-order}
 p_1^-,\ldots,p_m^-,p_m^+,\ldots,p_1^+,
\ee
where the signs refer to $v$ near the target saddle. We verify this order first as $e$ increases to $h_B$, when each crossing approaches its saddle, and then continue the crossings down to the chosen oval. Their order cannot change without an intersection. A stable branch may turn between the section and the saddle, so its sign on the section need not be its arrival sign.

We prove opposite monotonicity of the crossing phases. For a regularly normalized family $U_\alpha(y)$ of unstable or stable trajectories, we set $w\coloneqq dU_\alpha/d\alpha$ and $W\coloneqq\det(X_\alpha,w)$. The variation equation gives
\be
\label{eq:general-phase-variation}
 w_y=DX_\alpha w+(0,-v),\qquad W_y=-\alpha W-v^2.
\ee
As in \eqref{eq:shooting-transversality}, the saddle asymptotics give
\be
\label{eq:general-signed-determinants}
 W^u(y)=-e^{-\alpha y}\int_{-\infty}^y e^{\alpha s}v(s)^2\,ds<0,
 \qquad
 W^s(y)=e^{-\alpha y}\int_y^\infty e^{\alpha s}v(s)^2\,ds>0.
\ee
The positive source eigenvalue and the target eigenvalue $\lambda_-<-\alpha$ justify the respective boundary terms. Changing the time origin adds a multiple of $X_\alpha$ to $w$ and leaves $W$ unchanged. At a crossing with $v\neq0$, differentiation along the fixed oval yields
\be
\label{eq:general-phase-sign}
 W=\det(X_\alpha,X_0)\theta'(\alpha)=\alpha v^2\theta'(\alpha).
\ee
At a turning point, we instead use $W\neq0$: the map $(\alpha,y)\mapsto U_\alpha(y)$ is locally invertible. Its inverse parameter, restricted to $\Gamma_e$, satisfies
\be
\label{eq:general-phase-at-turn}
 \frac{d\alpha}{d\theta}=\frac{\alpha v^2}{W}.
\ee
This derivative has a fixed sign and vanishes only at the isolated turning points. Integration and inversion therefore give strict monotonicity through these points as well. The outgoing phase strictly decreases with $\alpha$; every stable phase strictly increases.

We lift the outgoing phase to a continuous real function $\vartheta$ on $I$, and the stable phases to continuous functions $\sigma_j$, $j\in\mathbb Z$, in the order~\eqref{eq:general-stable-order}, with $\sigma_{j+2m}=\sigma_j+T_e$. Their lifts remain strictly ordered. We define
\be
\label{eq:general-phase-differences}
 F_j\coloneqq\vartheta-\sigma_j,\qquad
 C_j\coloneqq\{\alpha\in I:F_j(\alpha)>0>F_{j+1}(\alpha)\}.
\ee
Each $F_j$ is continuous and strictly decreasing. A zero identifies the same point of an unstable and a stable branch at the same $\alpha$, and hence exactly a saddle connection. It is unique for this index. If $j<k$ have zeros, then $\beta_k<\beta_j$; for every $j<\ell<k$, we have $F_\ell(\beta_j)<0<F_\ell(\beta_k)$. Thus every intermediate index has its unique zero. Each $C_j$ is an open interval, possibly empty, and equals $(\beta_{j+1},\beta_j)$ when both endpoints exist. The phases are bounded on compact subsets of $I$; only finitely many period translates can meet $\vartheta$ there, proving local finiteness.

We check that these indices are the turns used in the article. At a crossing away from a turning point, let $N$ be the number of preceding turns of the outgoing trajectory. Its lift is the reduced phase plus $T_e\lfloor N/2\rfloor$. If a stable branch has $M$ turns from the section to its saddle, its lift is the reduced phase minus $T_e\lceil M/2\rceil$ for negative arrival, or minus $T_e\lfloor M/2\rfloor$ for positive arrival. These normalizations preserve the cyclic order: starting with $e$ close to $h_B$, where $M=0$, they continue without jumps when the section passes a turn. At a connection, $N+M=n$; the period difference is $T_e\lfloor n/2\rfloor$. The ranks of the negative and positive arrivals at $p_i$ are $i-1$ and $2m-i$. This proves~\eqref{eq:general-itinerary-index}. If the crossing itself is a turn, we vary $e$ slightly; the same continuous lifts give the same index. All preceding turns, including those before entry into $B$, are counted.

It remains to identify the intervals and prove exhaustion. At fixed $\alpha$, every point of $\Gamma_e$ except the $2m$ stable crossings converges to an equilibrium with energy below $-h_B$, and therefore enters a child bubble in finite time. Entry is open in the starting point and parameter, by finite-time continuous dependence and absorption. Consequently the child label is constant on each arc between consecutive stable crossings. On its two sides, a stable branch exits near its saddle along the two outgoing branches, as in the local estimate~\eqref{eq:saddle-exit-estimate}; these enter the two adjacent children. This gives the $D,R,L$ labels in the statement. Choosing a point continuously in each open arc and using openness of entry shows that its label is also constant as $\alpha$ varies. Every $\alpha\in I$ is a zero of some $F_j$ or belongs to exactly one $C_j$, proving the partition at this level.

We repeat on each nonempty $C_j$ in its child bubble. The number of interior saddles strictly decreases, so every descent ends after finitely many steps. There are countably many cells at each step and only finitely many levels; their critical sets together are at most countable. This proves exhaustion, including the case of the outer bubble, entered by every $\alpha>0$.

For the full critical set, we retain the countability conclusion obtained by recursion. We now verify that its complement is open. Near a sink at a fixed parameter $\alpha_*$, write the system as $z_y=A_\alpha z+O(|z|^2)$. The positive matrix
\be
\label{eq:attracting-ellipsoid-matrix}
 P\coloneqq\int_0^\infty e^{A_{\alpha_*}^Ts}e^{A_{\alpha_*}s}\,ds
\ee
satisfies $A_{\alpha_*}^TP+PA_{\alpha_*}=-\mathrm{Id}$. A sufficiently small ellipsoid $z^TPz\leq r^2$ is attracting and forward invariant for nearby parameters. Finite-time entry into it proves openness of each sink basin. Thus $\Sigma_B$ is relatively closed, and each component of its complement has a constant limiting sink. A connection has finitely many turns: it is monotone near its endpoint saddles, while every zero of $v$ in the remaining compact time interval has $v_y=q(u)\neq0$ and cannot accumulate. Simplicity follows from~\eqref{eq:shooting-transversality}. Uniqueness at the terminal step applies on the fixed entry interval supplied by the preceding steps.
\end{proof}

\section{Numerical experiments}
\label{sec:6}

Numerical experiments were essential to develop intuition and confirm hidden phenomena indicated by the theory.
In this section, we describe the numerical tools which we used to study travelling waves to the conservation
law~\eqref{c1.1} and which have already been used to produce some of the earlier figures.
All computations were performed in Python using the libraries
SciPy and Numpy, and the codes are available upon request.

We begin by observing that for the family of polynomial fluxes $f_K$ (where $K$ here denotes the odd polynomial degree) and a quadratic entropy,
it was possible to explicitly compute the entropy dissipation $E(\cdot, u_0)$
and the function $G(\cdot; u_0,a)$ using polynomial evaluations. In particular, 
the right hand side of the system \eqref{Agsystemb}, which only required the polynomial $\partial_u G$, 
could be computed analytically. The system \eqref{Agsystemb} was solved using a 
straightforward $4$-th order Runge-Kutta method with a constant timestep and
we detail later how the systems were initialized (or stopped) when they started (or
converged) asymptotically to a saddle equilibrium. A fixed timestep must resolve the saddle rates as well as the turns. 
The refinement check reported below tests the critical ratios; it is not a uniform pointwise error bound for all profiles. The accuracy for the polynomial $f_K$ 
allowed them to be used to validate our numerical techniques for
the computation of the entropy dissipation $E$ and the function $G$, appearing in the system \eqref{Agsystemb}, 
before they were applied to numerical fluxes, specifically fluxes given as cubic splines.

To easily produce a wide variety of fluxes with two continuous derivatives, we used the \texttt{CubicSpline}
function of the SciPy library of Python. In order to construct a flux with the same equilibria as $f_K$,
that is with the same intersection points between the graph of $f_K$ and the line $K(u-K) +f_K(K)$,
besides the values of $f_K$ at the intersection points at $k=1, \ldots,K$, we chose to set the remaining values 
of the cubic spline at the mid points
$$
      1+\frac{1}{2}, 2+\frac{1}{2}, \ldots,  K-\frac{1}{2}.
$$
Finally, we set the values of the first derivative at $u=1, K$ to complete the description of
the spline. When using the values of $f_K$ at these midpoints to determine the spline, we have found 
that the spline was almost visually indistinguishable from $f_K$.

For the fluxes given as cubic splines, it was necessary to numerically evaluate the entropy flux $F$,
the entropy dissipation $E$, and the function $G$, the three of which required numerical integrations.
We chose to use Simpson's composite quadrature, which is exact for polynomials of degree $3$. For the quadratic entropy, the integrands defining $F$, $E$, and $G$ are piecewise cubic. Splitting quadrature panels at every spline knot makes Simpson's rule exact in exact arithmetic; panels straddling knots require an error check by refinement. Continuity of two derivatives alone does not supply a quantitative error bound. These algorithms 
were validated by comparing their errors when computing 
these functions numerically, rather than analytically.

An independent check reconstructs the clamped cubic splines in Tables~\ref{tab:spline1} and \ref{tab:spline2} by solving for their nodal second derivatives. For $q(u)=f(u)-7u-2$, the only roots on $[1,7]$ are $1,2,\ldots,7$, all simple. Direct polynomial integration gives $G(1)<0<G(2)$ and $G(j)>0$ for $2\leq j\leq6$, with $G(7)=0$; thus the outer first root lies in $(1,2)$ and the relevant bounded region lies in the spline interval.

The check matches unstable and stable heights on a vertical section, rather than testing a finite-distance target near the final saddle. For monotone connections from $7$ to $5$ and from $5$ to $3$, the sections are $u=6$ and $u=4$. For the one-turn connection from $5$ to $3$, the returning positive branch is matched at $u=2$. Both branches start at displacement $\epsilon_0$ along a unit saddle eigenvector. The Runge--Kutta timestep and displacement are refined as follows; crossing times within a timestep are refined by bisection, and the parameter bracket is bisected to width below $2\cdot10^{-10}$.
\begin{center}
\begin{tabular}{cc|cc|ccc}
 timestep & $\epsilon_0$ & \multicolumn{2}{c|}{Table~\ref{tab:spline1}} & \multicolumn{3}{c}{Table~\ref{tab:spline2}}\\
 & & $7\to5$ & $5\to3$ & $7\to5$ & $5\to3$ & $5\to3$, one turn\\
\hline
 $0.002$ & $2\cdot10^{-5}$ &.549353403 &.573803633 &.212313149 &.644036807 &.217873130\\
 $0.001$ & $2\cdot10^{-5}$ &.549353396 &.573803633 &.212313147 &.644036807 &.217873130\\
 $0.0005$ & $2\cdot10^{-6}$ &.549353396 &.573803633 &.212313147 &.644036807 &.217873130
\end{tabular}
\end{center}
These refinements preserve the printed five-decimal values and the strict numerical gaps used in the examples. They do not give rigorous enclosures of the exact critical ratios.

We conclude this section by describing the algorithm we used to compute the
values of the diffusion-dispersion ratio for which there exists $n$-oscillation trajectories
connecting $u_j$ to $u_k$, both saddle equilibria. The first step though, is to
explain how, for a fixed value of $\alpha,$ the numerical solution to \eqref{Agsystemb} was initialized at $u_j$
and how asymptotic convergence to $u_k$ was determined. A small parameter $\epsilon_0$
was introduced, taken to be $2 \cdot 10^{-5}$ for the examples \ref{ex:counter1}
and \ref{ex:counter2}, and the solution $U$ was initialized to 
$$
   U(0) =  (u_j,0)^T + \epsilon_0 v_i(u_j),
$$
where $v_i$ is the unit eigenvector in $Q^i_j$.  Similarly, the finite-distance stopping test for approximate arrival at $u_k$ required a time $t^*$, after $n$ crossings of the $u$-axis, at which
$$
     \big\|  U(t^*) -  \big((u_k,0)^T + \epsilon_0 v_\ell(u_k)\big) \big\| < \delta_0, 
$$
for some preset tolerance $\delta_0$. To compute the numerical
value of the diffusion-dispersion ratio $\overline{\alpha}_n(j,k)$ for which
there exists an $n$-oscillation trajectory connecting $u_j$ to $u_k$, we implemented
a straightforward shooting method with a bisection method to converge
towards the optimal value of $\alpha$. The bisection method is based on 
computing the root of
$
       g(\alpha) := u_{\alpha}(t^*) - u_\ell 
$
where $u_\ell$ is the $u$-component of $ (u_k,0)^T + \epsilon_0 v_\ell(u_k) = (u_\ell,v_\ell)$
and $u_\alpha(t^*)$ is the component of the solution $U_\alpha(t^*)=(u_\alpha(t^*), v_\alpha(t^*))$
at the time $t^*$ at which the solution $U_\alpha$ reaches the height $v_\alpha(t^*) = v_\ell$.
Estimating $u_\alpha(t^*)$ required interpolations between the two timesteps during which the
numerical solution $U_\alpha$ crossed the height $v_\ell$.

\section{Concluding remarks}
\label{sec:conclusion}

The analysis in this paper exploited the fact that the entropy was given by the explicit expression \eqref{entropy:explicit},
specifically in the identity \eqref{relation_E_G}. Additional work would be needed for arbitrary convex entropies. A bound $U''(u)\geq c>0$ alone does not turn~\eqref{relation_E_G} or~\eqref{eq:explicitastar} into comparison inequalities, because the integrand $f(v)-f(u_0)-a(v-u_0)$ in the entropy formula can change sign. Furthermore, this work also treated only
linear regularizations of the form \eqref{defusidispersieq}, but as was done in \cite{BedChaCoqLeF}, the work
could be extended quite naturally to nonlinear regularizations. Such an extension would require suitable regularity, an absorbing energy, and a signed parameter variation identity for the chosen regularization, in addition to the analysis in~\cite{BedChaCoqLeF}.
Overall, we observe that a stronger selection criteria would be needed to handle
nonclassical shocks with two or more oscillations.

This research has also established that when three or more inflection points are present,
then new length scales appear that determine whether or not new equilibrium points
are reachable for small values of the diffusion-dispersion ratio. In the examples presented so
far, the equilibrium points were all separated by the same unit distance, but the variation in 
$G$ between the equilibrium decided whether some equilibria were reachable. A more
complete length scale would account for the distance between equilibria, but
in our analysis the attainability of equilibrium states was expressed entirely
in terms of the critical diffusion-dispersion ratios.

The next issue remains open: how to construct the Riemann solver for fluxes with an arbitrary number of
inflection points. The elementary classification and the exhaustive ordering for nested configurations describe families of profiles relevant to this problem. A global kinetic description and the construction of the corresponding Riemann solver require further analysis.

\paragraph{Acknowledgments}

This research was initiated when the second author (PLF) was a visiting research fellow at the Courant Institute for Mathematical Sciences, New York University. During the completion of this project, the second author were partially supported by the French research project ANR-23-CE40-0010-02 and the European research project MSCA-101131233.



\begin{thebibliography}{10}

\bibitem{BedChaCoqLeF} 
N. Bedjaoui, C. Chalons, F. Coquel, and P.G. LeFloch, 
Non-monotonic traveling waves in Van der Waals fluids, 
Analysis Appl. 3 (2005), 419--446. 

\bibitem{BedjaouiLeFloch1} N. Bedjaoui and P.G. LeFloch, 
Diffusive-dispersive traveling waves and kinetic relations. I. Nonconvex hyperbolic conservation laws, 
J. Differential Equations 178 (2002), 574--607.
 
\bibitem{BedjaouiLeFloch3} N. Bedjaoui and P.G. LeFloch, 
Diffusive-dispersive traveling waves and kinetic relations. III. An hyperbolic model of nonlinear elastodynamics,
Ann. Univ. Ferrara Sc. Mat. 47 (2001), 117--144.

\bibitem{BedjaouiLeFloch5} N. Bedjaoui and P.G. LeFloch, 
Diffusive-dispersive traveling waves and kinetic relations. V. Singular diffusion and dispersion terms, 
Proc. Royal Soc. Edinburgh 134A (2004), 815--844.  

\bibitem{ElHoeferShearer2017}
G.A. El, M.A. Hoefer, and M. Shearer,
Dispersive and diffusive-dispersive shock waves for nonconvex conservation laws,
SIAM Rev. 59 (2017), 3--61.

\bibitem{JMS} D. Jacobs, W. R. McKinney, and M. Shearer, Traveling wave solutions of the modified
Korteweg-de Vries-Burgers equation, J. Differential Equations 116 (1995), 448--467.

\bibitem{LaforestLeFloch} M. Laforest and P.G. LeFloch, 
Diminishing functionals for nonclassical entropy solutions selected by kinetic relations, 
Portugal Math. 67 (2010),  279--319.

\bibitem{LeFlochbook} P.G. LeFloch, 
{\sl Hyperbolic systems of conservation laws. The theory of classical and nonclassical shock waves,} 
Lectures in Mathematics, ETH Z\"urich, Birkh\"auser, 2002.

\bibitem{LeFlochreview} P.G. LeFloch, 
{\it Kinetic relations for undercompressive shock waves.
Physical, mathematical, and numerical issues,}
in ``Nonlinear Partial Differential Equations and Hyperbolic Wave Phenomena'', 
Contemporary Mathematics, vol. 526, 
Amer. Math. Soc., Providence, RI, 2010, pp.~237--272.
 

 


\end{thebibliography}
\end{document}